\documentclass[sn-mathphys,Numbered]{sn-jnl}
\usepackage[utf8]{inputenc}
\usepackage{xpatch}
\usepackage{fancyhdr}
\usepackage{latexsym}
\usepackage{eufrak}
\usepackage[T1]{fontenc}
\usepackage{tikz}
\usepackage{xparse}
\usepackage{enumitem}
\usepackage{indentfirst}
\usepackage{graphicx}
\usepackage{multirow}
\usepackage{amsmath,amssymb,amsfonts}
\usepackage{amsthm}
\usepackage{mathrsfs}
\usepackage[title]{appendix}
\usepackage{xcolor}
\usepackage{textcomp}
\usepackage{manyfoot}
\usepackage{booktabs}
\usepackage{algorithm}
\usepackage{algorithmicx}
\usepackage{algpseudocode}
\usepackage{listings}

\DeclareMathOperator*{\R}{Re}

\DeclareMathOperator*{\sgn}{sgn}
\theoremstyle{thmstyleone}
\newtheorem{definition}{Definition}[section]
\newtheorem{theorem}{Theorem}[section]
\newtheorem{lemma}{Lemma}[section]

\newtheorem{remark}{Remark}[section]

\newtheorem{proposition}{Proposition}[section]

\newtheorem{assumption}{Assumption}
\begin{document}
\title[Article Title]{On the fractional abstract Schr\"odinger-type evolution equations on the Hilbert space and its applications to the fractional dispersive equations}
\author*[1]{\fnm{Mingxuan} \sur{He}}\email{202121511131@smail.xtu.edu.cn}
\author[1]{\fnm{Na} \sur{Deng}}\email{202121511122@smail.xtu.edu.cn}
\equalcont{These authors contributed equally to this work.}
\affil*[1]{\orgdiv{School of Mathematics and Computational Science}, \orgname{Xiangtan University}, \orgaddress{\city{Xiangtan}, \postcode{411100}, \state{Hunan province}, \country{China}}}
\abstract{In this paper we prove the local and global well-posedness of the fractional abstract Schr\"odinger-type evolution equation ($iD_t^\alpha u+Au+F(u)=0$) on the Hilbert space and as an application, we prove the local and global well-posedness of the fractional dispersive equation with static potential ($D_t^\alpha-iP(D)u-iqu-iVu+F(u)=0$) under the only assumption that the symbol $P(\xi)$ of $P(D)$ behaves like $\left\lvert\xi\right\rvert^m$ for $\lvert\xi\rvert\to\infty$. In appendix, we also give the H\"older regularities and the asymptotic behaviors of the mild solution to the linear fractional abstract Schr\"odinger-type equation ($iD_t^\alpha u+Au+F(t)=0$). Because of the lack of the semigroup properties of the solution operators, we employ a strategy of proof based on the spectral theorem of the selfadjoint operators and the asymptotic behaviors of the Mittag-Leffler functions.}
\maketitle
\section{Introduction}
\subsection{Background and main results}
In the past few decades, fractional calculus has gained a great deal of attention from mathematicians and has been shown to be valuable in areas such as physics, engineering, and economics. For further information about fractional derivatives, we refer readers to \cite{Theory-and-Applications-of-Fractional-Differential-Equations-Chapter2-Fractional-integrals-and-fractional-derivatives,Fractional-Calculus-and-Waves-in-Linear-Viscoelasticity,Fractional-differential-equations-an-introduction-to-fractional-derivatives-fractional-differential-equations-to-methods-of-their-solution-and-some-of-their-applications} and we will give a brief introduction of them in  Appendix $\ref{9.1Fractionl integral and derivative definition appendix}$.

The goal of this paper is to investigate the well-posedness of the fractional abstract Schr\"odinger-type evolution equation
\begin{equation}
\begin{cases}
iD_t^\alpha u+Au+F(u)=0,\quad t>0\\
u(0)=x,
\end{cases}
\label{nonlinear Schrodinger equation Hilbert}
\end{equation}
on a separable Hilbert space $H$ with some suitable regularity hypotheses on $F$ and $x$. $A$ is a selfadjoint operator in $H$.

To state the hypotheses on $F$, we first introduce a function space $C_q[0,\infty)$: we say a continuous, nondecreasing and nonnegative function $w$ is in $C_q[0,\infty)$ if $w:[0,\infty)\to[0,\infty)$ satisfies $w(0)=0$ and $w(\sigma)\neq0$ when $\sigma\neq0$ and there exists a $\varepsilon>0$ such that
\[ \int_1^\infty\frac{\sigma^{\frac{1}{\alpha}+\varepsilon-1}}{w(\sigma)^{\frac{1}{\alpha}+\varepsilon}}d\sigma=q. \]
Hence we can state the hypotheses on $F$ as:
\begin{assumption}
$F(0)=0$. If $\left\lVert u(t)\right\rVert_{D(A)}$ and $\left\lVert v(t)\right\rVert_{D(A)}$ is bounded on $I\subset[0,\infty)$ a.e., then $\left\lVert F(u)-F(v)\right\rVert_{D(A)}\leq C\left\lVert u(t)-v(t)\right\rVert_{D(A)}$ a.e. on $I$ where $C$ is dependent on the initial data $u(0),v(0)$ with the norm $\rho(\cdot)$ and the essential upper bound of $\left\lVert u(t)\right\rVert_{D(A)}$ and $\left\lVert v(t)\right\rVert_{D(A)}$ on $I$.
\label{AssumptionA'}
\end{assumption}
\begin{assumption}
There exists a $w\in C_\infty[0,\infty)$ and a positive constant $C$ which depends on the initial data $u(0)$ with the norm $\varrho(\cdot)$ such that $\left\lVert F(u)\right\rVert_{D(A)}\leq Cw\left(\left\lVert u(t)\right\rVert_{D(A)}\right)$ pointwisely in $t$.
\label{AssumptionB'}
\end{assumption}
In Section $\ref{9.17Proof of Theorem unique mild solution X}$, $\ref{Proof of Theorem continuation and blow up alternative}$ and $\ref{Proof of Theorem unique global solution X}$, we will prove the following results.
\begin{theorem}[local well-posedness]
Let Assumption $\ref{AssumptionA'}$ hold and $x\in D(A)$ such that $\left\lvert x\right\rvert_1:=\max\left\{\left\lVert x\right\rVert_{D(A)},\rho(x)\right\}<\infty$. There exists a positive number $T$ which depends only on $\left\lVert x\right\rVert_{D(A)}$ and $\rho(x)$ such that $(\ref{nonlinear Schrodinger equation Hilbert})$ admits a unique strict solution $u(t)$ on $[0,T]$ in the class
\[ u\in C\left([0,T];D(A)\right),\mathbf{D}_t^\alpha(u-x)\in C\left([0,T];H\right). \]
Moreover, if $u(t), v(t)$ are the strict solutions of $(\ref{nonlinear Schrodinger equation Hilbert})$ with the initial data $x, y$ respectively, then there exists a positive constant $C$ which depends on $\rho(x),\rho(y)$ and $\left\lVert u\right\rVert_{L^\infty\left((0,T);D(A)\right)},\left\lVert v\right\rVert_{L^\infty\left((0,T);D(A)\right)}$ such that
\begin{equation}
\left\lVert u(t)-v(t)\right\rVert_{D(A)}\leq CE_{\alpha,1}\left(\Gamma(\alpha)t^\alpha\right)\left\lVert x-y\right\rVert_{D(A)}.
\label{unique mild solution X theorem equation1}
\end{equation}
\label{unique mild solution X}
\end{theorem}
\begin{theorem}[continuation and blow-up alternative]
Let the assumptions in Theorem $\ref{unique mild solution X}$ hold and $u$ be the strict solution of $(\ref{nonlinear Schrodinger equation Hilbert})$ on $[0,T]$. Then $u$ can be extended to a maximal interval $[0,T_{\max})$ uniquely such that
\[ u\in C\left([0,T_{\max});D(A)\right).\quad\mathbf{D}_t^\alpha(u-x)\in C\left([0,T_{\max});H\right) \]
and $T_{\max}<\infty$ implies $\lim\limits_{t\uparrow T_{\max}}\left\lVert u(t)\right\rVert_{D(A)}=\infty$.
\label{continuation and blow up alternative}
\end{theorem}
\begin{theorem}[global well-posedness]
Let the assumptions in Theorem $\ref{unique mild solution X}$ and Assumption $\ref{AssumptionB'}$ hold. If $x\in D(A)$ satisfying $\left\lvert x\right\rvert_2:=\max\left\{\left\lVert x\right\rVert_{D(A)},\rho(x),\varrho(x)\right\}<\infty$, $(\ref{nonlinear Schrodinger equation Hilbert})$ admits a unique strict solution $u(t)$ on $[0,\infty)$ in the class
\[ u\in C\left([0,\infty);D(A)\right),\quad\mathbf{D}_t^\alpha(u-x)\in C\left([0,\infty);H\right). \]
That is, the strict solution in Theorem $\ref{unique mild solution X}$ is global.
\label{unique global solution X}
\end{theorem}
\begin{remark}
You can find the notion of the solution of $(\ref{nonlinear Schrodinger equation Hilbert})$ in Definition $\ref{linear mild solution H}$, Definition $\ref{linear classical solution H}$ and Definition $\ref{linear strict solution H}$.
\end{remark}
In section $\ref{9.16Application. The well-posedness of the fractional dispersive equation}$, we shall show that these theorems are applicable to the very general fractional dispersive equation
\begin{equation}
\begin{cases}
D_t^\alpha u-iP(D)u-iqu-iVu+F(u)=0,\quad&x\in\mathbb{R}^n,\;t>0\\
u(0,x)=u_0(x),\quad&x\in\mathbb{R}^n
\end{cases}.
\label{9.1 very general dispersive equation motivated}
\end{equation}
Here $q\in L^2(\mathbb{R}^n)$ and $V\in L^\infty(\mathbb{R}^n)$ are both real-valued functions. $P(D)$ is defined via its real symbol, that is, $P(D)u=\mathscr{F}^{-1}\left(P(\xi)\mathscr{F}u\right)$ and $P(\xi)\in C\left(\mathbb{R}^n;\mathbb{R}\right)$ behaves like $\left\lvert\xi\right\rvert^m (m>\frac{n}{2})$ when $\lvert\xi\rvert\to\infty$. Here $\mathscr{F}$ denotes the Fourier transform and $\mathscr{F}^{-1}$ denotes the inverse Fourier transform. Note that no assumption is made on the behaviour of $P(\xi)$ for small $\xi$ except continuity. For some results in the integer order case ($\alpha=1$), you can see Constantin, Saut\cite{Local-smoothing-properties-of-dispersive-equations} and Kenig, Ponce, Vega\cite{Oscillatory-Integrals-and-Regularity-of-Dispersive-Equations}.
\begin{remark}
Specifically, as is easily seen that $(\ref{9.1 very general dispersive equation motivated})$ is the general case of some well-known different kinds of fractional disperve equations such as
\begin{align}
&iD_t^\alpha u+\left(-\Delta\right)^\beta u+q(x)u+V(x)u+\lambda\left\lvert u\right\rvert^{p-1}u=0,\;x\in\mathbb{R}^n,\;t>0,\label{space time fractional Schrodinger equation motivated}\\
&D_t^\alpha u+\partial_x^3u+u^m\partial_xu=0,\quad x\in\mathbb{R},\;t>0,\label{time fractional KdV equation motivated}\\
&D_t^\alpha u+H\partial_x^2u+u^m\partial_xu=0,\quad x\in\mathbb{R},\;t>0.\label{time fractional BO equation motivatied}
\end{align}
In $(\ref{space time fractional Schrodinger equation motivated})$, $\left(-\Delta\right)^\beta$ denotes the fractional Laplacian whose definition is $\left(-\Delta\right)^\beta u=\mathscr{F}^{-1}\left(\left\lvert\xi\right\rvert^{2\beta}\mathscr{F}u\right)$. In $(\ref{time fractional BO equation motivatied})$, $H$ denotes the Hilbert trasfrom whose definition is $Hu=\mathscr{F}^{-1}\left(i\sgn(\xi)\mathscr{F}u\right)$. Actually these equations have been studied by many authors but have not been studied in a more general and abstract way. $(\ref{time fractional KdV equation motivated})$ is called the time fractional m-gKdV equation and $(\ref{time fractional BO equation motivatied})$ is called the time fractional modified Benjamin-Ono equation (mBO equation). The researches on $(\ref{time fractional KdV equation motivated})$ and $(\ref{time fractional BO equation motivatied})$ mainly focused on solving them by variational iteration method\cite{Variational-iteration-method-for-solving-the-space-and-time-fractional-KdV-equation}, Adomian decomposition method\cite{Application-of-homotopy-perturbation-method-to-fractional-IVPs} and symmetry analysis\cite{Lie-symmetry-analysis-of-the-time-fractional-KdV-type-equation} and so on. Several works have been devoted to the well-posed problem for $(\ref{time fractional KdV equation motivated})$ and $(\ref{time fractional BO equation motivatied})$ in the integer order case ($\alpha=1$) which you can see \cite{On-the-Korteweg-de-Vries-equation, Well-Posedness-of-the-Initial-Value-Problem-for-the-Korteweg-de-Vries-Equation,On-the-local-well-posedness-of-the-Benjamin-Ono-and-modified-Benjamin-Ono-equations,On-the-global-well-posedness-of-the-Benjamin-Ono-equation}. There are much more studies on $(\ref{space time fractional Schrodinger equation motivated})$ which is called the space-time fractional nonlinear Schr\"odinger equation with static potential introduced by Achar, Yale, Hanneken\cite{Time-Fractional-Schrodinger-Equation-Revisited} in the case $q,V=0$ or $q=0$. Su, Zhao, Li\cite{Local-well-posedness-of-semilinear-space-time-fractional-Schrodinger-equation} have studied the local well-posedness of it by estimating the fundamental solution using the properties of $H$-functions. If $\beta=1$, it reduces to the time fractional nonlinear Schr\"odinger equation. Peng, Zhou, Ahmad\cite{The-well-posedness-for-fractional-nonlinear-Schrodinger-equations} have studied the global well-posedness of it by the decay estimates of the solution. Wang, Zhou, Wei\cite{Fractional-Schrodinger-equations-with-potential-and-optimal-controls} have studied the global well-posedness and some dynamical properties of it in a bounded domain. In particular, the integer order case ($\alpha=1$, $\beta=1$) has been studied extensively by mathematicians such as Kato\cite{On-nonlinear-Schrodinger-equations,On-nonlinear-Schrodinger-equations-II-HS-solutions-and-unconditional-well-posedness}, Cazenave\cite{Semilinear-Schrodinger-equations}, Ginibre, Velo\cite{On-a-class-of-nonlinear-Schrodinger-equations-I-The-Cauchy-problem-general-case,The-global-Cauchy-problem-for-the-nonlinear-Schrodinger-equation-revisited}, Bourgain\cite{Hyperbolic-Equations-and-Frequency-Interactions} and so on and the researches of the well-posedness of the space fractional case ($\alpha=1$) which is introduced by Laskin\cite{Fractional-Schrodinger-equation,Fractional-quantum-mechanics, Fractional-quantum-mechanics-and-Levy-path-integrals,Fractals-and-quantum-mechanics} you can see Guo, Han, Xin\cite{Existence-of-the-global-smooth-solution-to-the-period-boundary-value-problem-of-fractional-nonlinear-Schrodinger-equation}, Guo, Huo\cite{Global-Well-Posedness-for-the-Fractional-Nonlinear-Schrodinger-Equation} and Hong, Sire\cite{On-Fractional-Schrodinger-Equations-in-sobolev-spaces}. In addition, here are some more differences about the time fractionalisation of the Schr\"odinger equation whether we should fractionalize the constant $i$. In fact, Naber\cite{Time-fractional-Schrodinger-equation} use Wick rotation to raise a fractional power of order $\alpha$ of $i$ which turns out to be the classical Schr\"odinger equations with a time dependent Hamiltonian and Grande\cite{Space-Time-Fractional-Nonlinear-Schrodinger-Equation} studied the local well-posedness and local smoothing properties of it.
\end{remark}
In Section $\ref{9.17Linear estimate. The well-posedness of the linear equation}$ we will give some required estimates for the linear solution operators. In Appendix $\ref{9.1Fractionl integral and derivative definition appendix}$ and $\ref{9.17On the Mittag-Leffler functions}$ we will give some brief introduction of the fractional integrals and fractional derivatives and the Mittag-Leffler function. In Appendix $\ref{9.17The perturbation and the spectral theorem of the selfadjoint operators}$ the perturbation of the selfadjoint operators will be stated and we will prove a general spectral theorem of the selfadjoint operators for the purpose of estimating the linear solution operators and proving Theorem $\ref{unique mild solution X}$ to Theorem $\ref{unique global solution X}$. In Appendix $\ref{Some further results of the linear}$ some futher results of the linear form of $(\ref{nonlinear Schrodinger equation Hilbert})$ will be given such as H\"older regularities and asymptotic behaviors.
\subsection{Notations}
The following notations are used without particular comments.
\begin{align*}
&L_T^\infty H=L^\infty\left((0,T);H\right),\quad L_T^\infty D(A)=L^\infty\left((0,T);D(A)\right),\\
&C_T^\alpha H=C^\alpha\left([0,T];H\right),\quad C_{[\delta,T]}^\alpha H=C^\alpha\left([\delta,T];H\right),\\
&L_t^\infty H=L^\infty\left((0,\infty);H\right),\quad L_t^\infty D(A)=L^\infty\left((0,\infty);D(A)\right),\\
&L_{(T_1,T_2)}^\infty H=L^\infty\left((T_1,T_2);H\right),\quad L_{(T_1,T_2)}^\infty D(A)=L^\infty\left((T_1,T_2);D(A)\right).
\end{align*}

We denote by $a\lesssim b$ if there exists a positive number $C$ which is independent on $\varepsilon$ (see the definition of $C_q[0,\infty)$), $T$ (local in time), the norm of the initial data ($\rho(\cdot)$,$\varrho(\cdot)$) and the essential upper bound of $\left\lVert u(t)\right\rVert_{D(A)}$, $\left\lVert v(t)\right\rVert_{D(A)}$ (see Assumption $\ref{AssumptionA'}$) such that $a\leq Cb$. We denote by $a\sim b$ if $b\lesssim a\lesssim b$. We say $u$ is in a ball with radius $R$ of $Z$ if $u\in Z$ satisfies $\left\lVert u\right\rVert_Z\leq R$. We denote by $*$ the convolution in time, that is,
\[ u(t)*v(t)=\int_0^tu(t-\tau)v(\tau)d\tau. \]
We denote by $\vee$ the maximum and $\wedge$ the minimum.
\label{9.1 subsection Notations}
\section{Linear estimate. The well-posedness of the linear equation}
\label{9.17Linear estimate. The well-posedness of the linear equation}
We will call it the linear $(\ref{nonlinear Schrodinger equation Hilbert})$ if $F(u)=F(t)$ in $(\ref{nonlinear Schrodinger equation Hilbert})$. In this section, we shall give some estimates of the solution operator to the linear $(\ref{nonlinear Schrodinger equation Hilbert})$ and consider the well-posedness of it. More results of the linear $(\ref{nonlinear Schrodinger equation Hilbert})$ will be given in Appendix $\ref{Some further results of the linear}$. By the work of Zhou, Peng, Huang\cite{Duhamels-formula-for-time-fractional-Schrodinger-equations}, the solution of the linear $(\ref{nonlinear Schrodinger equation Hilbert})$ is given by
\begin{equation}
u(t)=S_tx+iGF(t)
\label{mild solution1}
\end{equation}
where
\[ Gv(t)=\int_0^tP_{t-\tau}v(\tau)d\tau \]
and
\begin{align*}
&S_t\phi=U\left(a(t,\xi)U^{-1}\phi\right),\quad a(t,\xi)=E_{\alpha,1}\left(ia(\xi)t^\alpha\right),\\
&P_t\phi=U\left(b(t,\xi)U^{-1}\phi\right),\quad b(t,\xi)=t^{\alpha-1}E_{\alpha,\alpha}\left(ia(\xi)t^\alpha\right).
\end{align*}
Note that $\frac{d}{dt}a(t,\xi)=ia(\xi)b(t,\xi)$ and $\frac{d}{dt}b(t,\xi)=t^{\alpha-2}E_{\alpha,\alpha-1}\left(ia(\xi)t^\alpha\right)$ (Theorem $\ref{Derivative Mittag Leffler}$). Their method is based on the spectral theorem of the selfadjoint operator (see appendix $\ref{9.5appendix The spectral theorem of the selfadjoint operator}$) and we can define the mild solution, the classical solution and the strict solution by this way as follows:
\begin{definition}
For $T>0$, let $x\in H$. The function $u\in C\left([0,T];H\right)$ given by $(\ref{mild solution1})$ is the mild solution of the linear $(\ref{nonlinear Schrodinger equation Hilbert})$ on $[0,T]$
\label{linear mild solution H}
\end{definition}
\begin{definition}
For $T>0$, a function $u:[0,T]\to H$ is a classical solution of the linear $(\ref{nonlinear Schrodinger equation Hilbert})$ on $[0,T]$ if $u$ in the class
\[ u\in C\left((0,T];D(A)\right)\cap C\left([0,T];H\right),\mathbf{D}_t^\alpha(u-u(0))\in C\left((0,T];H\right) \]
satisfies the linear $(\ref{nonlinear Schrodinger equation Hilbert})$.
\label{linear classical solution H}
\end{definition}
\begin{definition}
For $T>0$, a function $u:[0,T]\to H$ is a strict solution of the linear $(\ref{nonlinear Schrodinger equation Hilbert})$ on $[0,T]$ if $u$ in the class
\[ u\in C\left([0,T];D(A)\right),\mathbf{D}_t^\alpha(u-u(0))\in C\left([0,T];H\right) \]
satisfies the linear $(\ref{nonlinear Schrodinger equation Hilbert})$.
\label{linear strict solution H}
\end{definition}
Let $\chi_t=\chi_t(\xi)=\chi_{t^\alpha\lvert a(\xi)\rvert\leq M}$ and $\chi_t^c=\chi_t^c(\xi)=1-\chi_t$ where $M$ is large enough. Here $\chi_t$ denotes a smooth function supported on the set $\left\{(t,\xi):t^\alpha\lvert a(\xi)\rvert\leq2M\right\}$ satisfying $\chi_t=1$ if $t^\alpha\lvert a(\xi)\rvert\leq M$ and hence $\chi_t^c=\chi_{t^\alpha\lvert a(\xi)\rvert>2M}$ is a smooth function supported on the set $\left\{(t,\xi):t^\alpha\lvert a(\xi)\rvert>M\right\}$ satisfying $\chi_t^c=1$ if $t^\alpha\lvert a(\xi)\rvert>2M$.

We can now define the following operators:
\begin{align*}
&S_t^l\phi:=U\left(\chi_ta(t,\xi)U^{-1}\phi\right),\quad S_t^h\phi:=U\left(\chi_t^ca(t,\xi)U^{-1}\phi\right),\\
&P_t^l\phi:=U\left(\chi_tb(t,\xi)U^{-1}\phi\right),\quad P_t^h\phi:=U\left(\chi_t^cb(t,\xi)U^{-1}\phi\right),\\
&G^lv(t):=\int_0^tP_{t-\tau}^lv(\tau)d\tau,\quad G^hv(t):=\int_0^tP_{t-\tau}^hv(\tau)d\tau.
\end{align*}
According to Theorem $\ref{Mittag-Leffler function asymptotic expansion}$, it follows that
\begin{align*}
&\chi_t^ca(t,\xi)=\frac{i}{\Gamma(1-\alpha)}\chi_t^ca(\xi)^{-1}t^{-\alpha}+\chi_t^cO\left(\left\lvert a(\xi)\right\rvert^{-2}t^{-2\alpha}\right),\\
&\chi_t^cb(t,\xi)=\frac{1}{\Gamma(-\alpha)}\chi_t^ca(\xi)^{-2}t^{-\alpha-1}+\chi_t^cO\left(\left\lvert a(\xi)\right\rvert^{-3}t^{-2\alpha-1}\right),
\end{align*}
which then implies that
\begin{align*}
&\begin{aligned}
S_t^h\phi&=\frac{i}{\Gamma(1-\alpha)}t^{-\alpha}U\left(a(\xi)^{-1}\chi_t^cU^{-1}\phi\right)+U\left(O\left(\left\lvert a(\xi)\right\rvert^{-2}t^{-2\alpha}\right)\chi_t^cU^{-1}\phi\right)\\
&=:\frac{i}{\Gamma(1-\alpha)}t^{-\alpha}\mathbf{A}_t^{-1}\phi+R_t^S\phi,
\end{aligned}\\
&\begin{aligned}
P_t^h\phi&=\frac{1}{\Gamma(-\alpha)}t^{-\alpha-1}U\left(a(\xi)^{-2}\chi_t^cU^{-1}\phi\right)+U\left(O\left(\left\lvert a(\xi)\right\rvert^{-3}t^{-2\alpha-1}\right)\chi_t^cU^{-1}\phi\right)\\
&=:\frac{1}{\Gamma(-\alpha)}t^{-\alpha-1}\mathbf{A}_t^{-2}\phi+R_t^P\phi,
\end{aligned}\\
&G^hv(t)=\frac{1}{\Gamma(-\alpha)}\int_0^t\left(t-\tau\right)^{-\alpha-1}\mathbf{A}_{t-\tau}^{-2}v(\tau)d\tau+\int_0^tR_{t-\tau}^Pv(\tau)d\tau.
\end{align*}

Note that the following relations hold:
\[ S_t\phi=S_t^l\phi+S_t^h\phi,\quad P_t\phi=P_t^l\phi+P_t^h\phi,\quad Gv(t)=G^lv(t)+G^hv(t). \]

Here we first give some estimates of the operator $\mathbf{A}_t^{-1}$, $\mathbf{A}_t^{-2}$, $R_t^S$ and $R_t^P$.
\begin{lemma}
$\mathbf{A}_t^{-1}$ maps $H$ into $H$ boundedly for every $t\geq0$ with the estimate
\begin{equation}
\left\lVert\mathbf{A}_t^{-1}\phi\right\rVert_H\lesssim t^\alpha\left\lVert\phi\right\rVert_H,
\label{lemma linear new approach1 equation1}
\end{equation}
and $A\mathbf{A}_t^{-1}$ maps $H$ into $H$ boundedly for every $t\geq0$ with the estimate
\begin{equation}
\left\lVert A\mathbf{A}_t^{-1}\phi\right\rVert_H\leq\left\lVert \phi\right\rVert_H.
\label{lemma linear new approach1 equation2}
\end{equation}
\label{lemma linear new approach1}
\end{lemma}
\begin{proof}[Proof]
$(\ref{lemma linear new approach1 equation1})$ can be easily proved using the fact $\left\lvert\chi_t^ca(\xi)^{-1}\right\rvert\lesssim t^\alpha$ and $(\ref{lemma linear new approach1 equation2})$ can be easily proved using the fact $\left\lvert\chi_t^c\right\rvert\leq1$.
\end{proof}
By the same way, we can easily proove the following lemmas.
\begin{lemma}
$\mathbf{A}_t^{-2}$ maps $H$ into $H$ boundedly for every $t\geq0$ with the estimate
\begin{equation}
\left\lVert\mathbf{A}_t^{-2}\phi\right\rVert_H\lesssim t^{2\alpha}\left\lVert\phi\right\rVert_H,
\label{lemma linear new approach3 equation1}
\end{equation}
and $A\mathbf{A}_t^{-2}$ maps $H$ into $H$ boundedly for every $t\geq0$ with the estimate
\begin{equation}
\left\lVert A\mathbf{A}_t^{-2}\phi\right\rVert_H\lesssim t^\alpha\left\lVert\phi\right\rVert_H.
\label{lemma linear new approach3 equation2}
\end{equation}
\label{lemma linear new approach3}
\end{lemma}
\begin{lemma}
$R_t^S$ maps $H$ into $H$ boundedly for every $t\geq0$ with the estimate
\begin{equation}
\left\lVert R_t^S\phi\right\rVert_H\lesssim\left\lVert\phi\right\rVert_H,
\label{lemma linear new approach2 equation1}
\end{equation}
and $AR_t^S$ maps $H$ into $H$ boundedly for every $t>0$ with the estimate
\begin{equation}
\left\lVert AR_t^S\phi\right\rVert_H\lesssim t^{-\alpha}\left\lVert\phi\right\rVert_H.
\label{lemma linear new approach2 equation2}
\end{equation}
\label{lemma linear new approach2}
\end{lemma}
\begin{lemma}
$R_t^P$ maps $H$ into $H$ boundedly for every $t>0$ with the estimate
\begin{equation}
\left\lVert R_t^P\phi\right\rVert_H\lesssim t^{\alpha-1}\left\lVert\phi\right\rVert_H.
\label{lemma linear new approach4 equation1}
\end{equation}
\label{lemma linear new approach4}
\end{lemma}
\begin{lemma}
Let $\phi\in H$. For any $t,s>0$, we have
\begin{equation}
\left\lVert t^{-\alpha}\mathbf{A}_t^{-1}\phi-s^{-\alpha}\mathbf{A}_s^{-1}\phi\right\rVert_H\lesssim\left(1+\left(t\wedge s\right)^{-1}\right)\lvert t-s\rvert\left\lVert\phi\right\rVert_H,
\label{lemma linear new approach bilinear1 equation1}
\end{equation}
and
\begin{equation}
\left\lVert t^{-\alpha}A\mathbf{A}_t^{-1}\phi-s^{-\alpha}A\mathbf{A}_s^{-1}\phi\right\rVert_H\lesssim\left(\left\lvert t^{-\alpha}-s^{-\alpha}\right\rvert+\left\lvert t^{1-\alpha}-s^{1-\alpha}\right\rvert\right)\left\lVert\phi\right\rVert_H.
\label{lemma linear new approach bilinear1 equation2}
\end{equation}
If moreover $\phi\in D(A)$, we have
\begin{equation}
\left\lVert t^{-\alpha}\mathbf{A}_t^{-1}\phi-s^{-\alpha}\mathbf{A}_s^{-1}\phi\right\rVert_H\lesssim\left(\left\lvert t^\alpha-s^\alpha\right\rvert+\left\lvert t^{\alpha+1}-s^{\alpha+1}\right\rvert\right)\left\lVert\phi\right\rVert_{D(A)}.
\label{lemma linear new approach bilinear1 equation3}
\end{equation}
\label{lemma linear new approach bilinear1}
\end{lemma}
\begin{proof}[Proof]
Note that $\left\lvert t^{-\alpha-1}\chi_t^c\right\rvert\lesssim t^{-1}\left\lvert a(\xi)\right\rvert$, $\left\lvert t^{-\alpha-1}\chi_t^c\right\rvert\lesssim t^{\alpha-1}\left\lvert a(\xi)\right\rvert^2$, $\left\lvert t^{-\alpha}\frac{d}{dt}\chi_t^c\right\rvert\lesssim\lvert a(\xi)\rvert$ and $\left\lvert t^{-\alpha}\frac{d}{dt}\chi_t^c\right\rvert\lesssim t^\alpha\left\lvert a(\xi)\right\rvert^2$. It follows that
\begin{align*}
\left\lvert t^{-\alpha}\chi_t^c-s^{-\alpha}\chi_s^c\right\rvert&=\left\lvert\int_s^t\frac{d}{d\tau}\left(\tau^{-\alpha}\chi_\tau^c\right)d\tau\right\rvert\\
&=\left\lvert\int_s^t-\alpha\tau^{-\alpha-1}\chi_\tau^c+\tau^{-\alpha}\frac{d}{d\tau}\chi_\tau^cd\tau\right\rvert\\
&\lesssim\left\lvert\int_s^t\lvert a(\xi)\rvert\tau^{-1}+\lvert a(\xi)\rvert d\tau\right\rvert\\
&\leq\left(1+\left(t\wedge s\right)^{-1}\right)\lvert a(\xi)\rvert\lvert t-s\rvert
\end{align*}
and
\begin{align*}
\left\lvert t^{-\alpha}\chi_t^c-s^{-\alpha}\chi_s^c\right\rvert&=\left\lvert\int_s^t-\alpha\tau^{-\alpha-1}\chi_\tau^c+\tau^{-\alpha}\frac{d}{d\tau}\chi_\tau^cd\tau\right\rvert\\
&\lesssim\left\lvert\int_s^t\tau^{\alpha-1}\left\lvert a(\xi)\right\rvert^2+\tau^\alpha\left\lvert a(\xi)\right\rvert^2d\tau\right\rvert\\
&\lesssim\left(\left\lvert t^\alpha-s^\alpha\right\rvert+\left\lvert t^{\alpha+1}-s^{\alpha+1}\right\rvert\right)\left\lvert a(\xi)\right\rvert^2.
\end{align*}
Then we have
\begin{align*}
\left\lVert t^{-\alpha}\mathbf{A}_t^{-1}\phi-s^{-\alpha}\mathbf{A}_s^{-1}\phi\right\rVert_H&=\left\lVert t^{-\alpha}a(\xi)^{-1}\chi_t^cU^{-1}\phi-s^{-\alpha}a(\xi)^{-1}\chi_s^cU^{-1}\phi\right\rVert_{L^2(\Omega)}\\
&\lesssim\left(1+\left(t\wedge s\right)^{-1}\right)\lvert t-s\rvert\left\lVert\phi\right\rVert_H,
\end{align*}
and
\begin{align*}
\left\lVert t^{-\alpha}\mathbf{A}_t^{-1}\phi-s^{-\alpha}\mathbf{A}_s^{-1}\phi\right\rVert_H&=\left\lVert t^{-\alpha}a(\xi)^{-1}\chi_t^cU^{-1}\phi-s^{-\alpha}a(\xi)^{-1}\chi_s^cU^{-1}\phi\right\rVert_{L^2(\Omega)}\\
&\lesssim\left(\left\lvert t^\alpha-s^\alpha\right\rvert+\left\lvert t^{\alpha+1}-s^{\alpha+1}\right\rvert\right)\left\lVert\phi\right\rVert_{D(A)}.
\end{align*}
Moreover, by the estimate
\begin{align*}
\left\lvert t^{-\alpha}\chi_t^c-s^{-\alpha}\chi_s^c\right\rvert&=\left\lvert\int_s^t-\alpha\tau^{-\alpha-1}\chi_\tau^c+\tau^{-\alpha}\frac{d}{d\tau}\chi_\tau^cd\tau\right\rvert\\
&\lesssim\left\lvert\int_s^t\tau^{-\alpha-1}+\tau^{-\alpha}d\tau\right\rvert\\
&\lesssim\left\lvert t^{-\alpha}-s^{-\alpha}\right\rvert+\left\lvert t^{1-\alpha}-s^{1-\alpha}\right\rvert
\end{align*}
we obtain
\begin{align*}
\left\lVert t^{-\alpha}A\mathbf{A}_t^{-1}\phi-s^{-\alpha}A\mathbf{A}_s^{-1}\phi\right\rVert_H&=\left\lVert t^{-\alpha}\chi_t^cU^{-1}\phi-s^{-\alpha}\chi_s^cU^{-1}\phi\right\rVert_{L^2(\Omega)}\\
&\lesssim\left(\left\lvert t^{-\alpha}-s^{-\alpha}\right\rvert+\left\lvert t^{1-\alpha}-s^{1-\alpha}\right\rvert\right)\left\lVert\phi\right\rVert_H.
\end{align*}
\end{proof}
\begin{lemma}
Let $\phi\in H$. For any $t,s>0$, we have
\begin{equation}
\left\lVert t^{-\alpha-1}\mathbf{A}_t^{-2}\phi-s^{-\alpha-1}\mathbf{A}_s^{-2}\phi\right\rVert_H\lesssim\left(\left\lvert t^{\alpha-1}-s^{\alpha-1}\right\rvert+\left\lvert t^\alpha-s^\alpha\right\rvert\right)\left\lVert\phi\right\rVert_H.
\label{8.16lemma linear new approach bilinear3 equation1}
\end{equation}
\label{8.16lemma linear new approach bilinear3}
\end{lemma}
\begin{proof}[Proof]
Note that $\left\lvert t^{-\alpha-2}\chi_t^c\right\rvert\lesssim t^{\alpha-2}\left\lvert a(\xi)\right\rvert^2$ and $\left\lvert t^{-\alpha-1}\frac{d}{dt}\chi_t^c\right\rvert\lesssim t^{\alpha-1}\left\lvert a(\xi)\right\rvert^2$. It follows that
\begin{align*}
\left\lvert t^{-\alpha-1}\chi_t^c-s^{-\alpha-1}\chi_s^c\right\rvert&=\left\lvert\int_s^t\frac{d}{d\tau}\left(\tau^{-\alpha-1}\chi_\tau^c\right)d\tau\right\rvert\\
&=\left\lvert\int_s^t(-\alpha-1)\tau^{-\alpha-2}\chi_\tau^c+\tau^{-\alpha-1}\frac{d}{d\tau}\chi_\tau^cd\tau\right\rvert\\
&\lesssim\left\lvert\int_s^t\tau^{\alpha-2}\left\lvert a(\xi)\right\rvert^2+\tau^{\alpha-1}\left\lvert a(\xi)\right\rvert^2d\tau\right\rvert\\
&\lesssim\left(\left\lvert t^{\alpha-1}-s^{\alpha-1}\right\rvert+\left\lvert t^\alpha-s^\alpha\right\rvert\right)\left\lvert a(\xi)\right\rvert^2.
\end{align*}
Hence we can obtain
\begin{align*}
\left\lVert t^{-\alpha-1}\mathbf{A}_t^{-2}\phi-s^{-\alpha-1}\mathbf{A}_s^{-2}\phi\right\rVert_H&=\left\lVert t^{-\alpha-1}a(\xi)^{-2}\chi_t^cU^{-1}\phi-s^{-\alpha-1}a(\xi)^{-2}\chi_s^cU^{-1}\phi\right\rVert_{L^2(\Omega)}\\
&\lesssim\left(\left\lvert t^{\alpha-1}-s^{\alpha-1}\right\rvert+\left\lvert t^\alpha-s^\alpha\right\rvert\right)\left\lVert\phi\right\rVert_H.
\end{align*}
\end{proof}
\begin{lemma}
Let $\phi\in H$. For any $t,s>0$, we have
\begin{equation}
\left\lVert R_t^S\phi-R_s^S\phi\right\rVert_H\lesssim\left(1+\left(t\wedge s\right)^{-1}\right)\lvert t-s\rvert\left\lVert\phi\right\rVert_H,
\label{lemma linear new approach bilinear2 equation2}
\end{equation}
and
\begin{equation}
\left\lVert AR_t^S\phi-AR_s^S\phi\right\rVert_H\lesssim\left(\left\lvert t^{-\alpha}-s^{-\alpha}\right\rvert+\left\lvert t^{1-\alpha}-s^{1-\alpha}\right\rvert\right)\left\lVert\phi\right\rVert_H.
\label{lemma linear new approach bilinear2 equation3}
\end{equation}
If moreover $\phi\in D(A)$, we have
\begin{equation}
\left\lVert R_t^S\phi-R_s^S\phi\right\rVert_H\lesssim\left(\left\lvert t^\alpha-s^\alpha\right\rvert+\left\lvert t^{\alpha+1}-s^{\alpha+1}\right\rvert\right)\left\lVert\phi\right\rVert_{D(A)}.
\label{lemma linear new approach bilinear2 equation4}
\end{equation}
\label{lemma linear new approach bilinear2}
\end{lemma}
\begin{proof}[Proof]
Using the fact $\left\lvert t^{-2\alpha-1}\chi_t^c\right\rvert\lesssim t^{-1}\left\lvert a(\xi)\right\rvert^2$, $\left\lvert t^{-2\alpha-1}\chi_t^c\right\rvert\lesssim t^{\alpha-1}\left\lvert a(\xi)\right\rvert^3$, $\left\lvert t^{-2\alpha}\frac{d}{dt}\chi_t^c\right\rvert\lesssim\left\lvert a(\xi)\right\rvert^2$ and $\left\lvert t^{-2\alpha}\frac{d}{dt}\chi_t^c\right\rvert\lesssim t^\alpha\left\lvert a(\xi)\right\rvert^3$ we obtain
\begin{align*}
\left\lvert t^{-2\alpha}\chi_t^c-s^{-2\alpha}\chi_s^c\right\rvert&=\left\lvert\int_s^t\frac{d}{d\tau}\left(\tau^{-2\alpha}\chi_\tau^c\right)d\tau\right\rvert\\
&=\left\lvert\int_s^t-2\alpha\tau^{-2\alpha-1}\chi_\tau^c+\tau^{-2\alpha}\frac{d}{d\tau}\chi_\tau^cd\tau\right\rvert\\
&\lesssim\left\lvert\int_s^t\tau^{-1}\left\lvert a(\xi)\right\rvert^2+\left\lvert a(\xi)\right\rvert^2d\tau\right\rvert\\
&\leq\left(1+\left(t\wedge s\right)^{-1}\right)\left\lvert a(\xi)\right\rvert^2\lvert t-s\rvert,
\end{align*}
and
\begin{align*}
\left\lvert t^{-2\alpha}\chi_t^c-s^{-2\alpha}\chi_s^c\right\rvert&=\left\lvert\int_s^t-2\alpha\tau^{-2\alpha-1}\chi_\tau^c+\tau^{-2\alpha}\frac{d}{d\tau}\chi_\tau^cd\tau\right\rvert\\
&\lesssim\left\lvert\int_s^t\tau^{\alpha-1}\left\lvert a(\xi)\right\rvert^3+\tau^\alpha\left\lvert a(\xi)\right\rvert^3d\tau\right\rvert\\
&\lesssim\left(\left\lvert t^\alpha-s^\alpha\right\rvert+\left\lvert t^{\alpha+1}-s^{\alpha+1}\right\rvert\right)\left\lvert a(\xi)\right\rvert^3,
\end{align*}
which then implies that
\begin{align*}
\left\lVert R_t^S\phi-R_s^S\phi\right\rVert_H&=\left\lVert O\left(\left\lvert a(\xi)\right\rvert^{-2}\right)t^{-2\alpha}\chi_t^cU^{-1}\phi-O\left(\left\lvert a(\xi)\right\rvert^{-2}\right)s^{-2\alpha}\chi_s^cU^{-1}\phi\right\rVert_{L^2(\Omega)}\\
&\lesssim\left(1+\left(t\wedge s\right)^{-1}\right)\lvert t-s\rvert\left\lVert\phi\right\rVert_H
\end{align*}
and
\begin{align*}
\left\lVert R_t^S\phi-R_s^S\phi\right\rVert_H&=\left\lVert O\left(\left\lvert a(\xi)\right\rvert^{-2}\right)t^{-2\alpha}\chi_t^cU^{-1}\phi-O\left(\left\lvert a(\xi)\right\rvert^{-2}\right)s^{-2\alpha}\chi_s^cU^{-1}\phi\right\rVert_{L^2(\Omega)}\\
&\lesssim\left(\left\lvert t^\alpha-s^\alpha\right\rvert+\left\lvert t^{\alpha+1}-s^{\alpha+1}\right\rvert\right)\left\lVert\phi\right\rVert_{D(A)}.
\end{align*}
Note that $\left\lvert t^{-2\alpha-1}\chi_t^c\right\rvert\lesssim t^{-\alpha-1}\left\lvert a(\xi)\right\rvert$ and $\left\lvert t^{-2\alpha}\frac{d}{dt}\chi_t^c\right\rvert\lesssim t^{-\alpha}\left\lvert a(\xi)\right\rvert$. It follows that
\begin{align*}
\left\lvert t^{-2\alpha}\chi_t^c-s^{-2\alpha}\chi_s^c\right\rvert&=\left\lvert\int_s^t-2\alpha\tau^{-2\alpha-1}\chi_\tau^c+\tau^{-2\alpha}\frac{d}{d\tau}\chi_\tau^cd\tau\right\rvert\\
&\lesssim\left\lvert\int_s^t\tau^{-\alpha-1}\left\lvert a(\xi)\right\rvert+\tau^{-\alpha}\left\lvert a(\xi)\right\rvert d\tau\right\rvert\\
&\lesssim\left(\left\lvert t^{-\alpha}-s^{-\alpha}\right\rvert+\left\lvert t^{1-\alpha}-s^{1-\alpha}\right\rvert\right)\left\lvert a(\xi)\right\rvert.
\end{align*}
Then we have
\begin{align*}
\left\lVert AR_t^S\phi-AR_s^S\phi\right\rVert_H&=\left\lVert a(\xi)O\left(\left\lvert a(\xi)\right\rvert^{-2}\right)t^{-2\alpha}\chi_t^cU^{-1}\phi-a(\xi)O\left(\left\lvert a(\xi)\right\rvert^{-2}\right)s^{-2\alpha}\chi_s^cU^{-1}\phi\right\rVert_{L^2(\Omega)}\\
&\lesssim\left(\left\lvert t^{-\alpha}-s^{-\alpha}\right\rvert+\left\lvert t^{1-\alpha}-s^{1-\alpha}\right\rvert\right)\left\lVert\phi\right\rVert_H.
\end{align*}
\end{proof}
\begin{lemma}
Let $\phi\in H$. For any $t,s>0$, we have
\begin{equation}
\left\lVert R_t^P\phi-R_s^P\phi\right\rVert_H\lesssim\left(\left\lvert t^{\alpha-1}-s^{\alpha-1}\right\rvert+\left\lvert t^\alpha-s^\alpha\right\rvert\right)\left\lVert\phi\right\rVert_H.
\label{8.16lemma linear new approach bilinear4 equation1}
\end{equation}
\label{8.16lemma linear new approach bilinear4}
\end{lemma}
\begin{proof}[Proof]
Using the fact $\left\lvert t^{-2\alpha-2}\chi_t^c\right\rvert\lesssim t^{\alpha-2}\left\lvert a(\xi)\right\rvert^3$ and $\left\lvert t^{-2\alpha-1}\frac{d}{dt}\chi_t^c\right\rvert\lesssim t^{\alpha-1}\left\lvert a(\xi)\right\rvert^3$ we obtain
\begin{align*}
\left\lvert t^{-2\alpha-1}\chi_t^c-s^{-2\alpha-1}\chi_s^c\right\rvert&=\left\lvert\int_s^t\frac{d}{d\tau}\left(\tau^{-2\alpha-1}\chi_\tau^c\right)d\tau\right\rvert\\
&=\left\lvert\int_s^t(-2\alpha-1)\tau^{-2\alpha-2}\chi_\tau^c+\tau^{-2\alpha-1}\frac{d}{d\tau}\chi_\tau^cd\tau\right\rvert\\
&\lesssim\left\lvert\int_s^t\left(\tau^{\alpha-2}+\tau^{\alpha-1}\right)\left\lvert a(\xi)\right\rvert^3d\tau\right\rvert\\
&\lesssim\left(\left\lvert t^{\alpha-1}-s^{\alpha-1}\right\rvert+\left\lvert t^\alpha-s^\alpha\right\rvert\right)\left\lvert a(\xi)\right\rvert^3.
\end{align*}
Hence there holds
\begin{align*}
\left\lVert R_t^P\phi-R_s^P\phi\right\rVert_H&=\left\lVert O\left(\left\lvert a(\xi)\right\rvert^{-3}\right)t^{-2\alpha-1}\chi_t^cU^{-1}\phi-O\left(\left\lvert a(\xi)\right\rvert^{-3}\right)s^{-2\alpha-1}\chi_s^cU^{-1}\phi\right\rVert_{L^2(\Omega)}\\
&\lesssim\left(\left\lvert t^{\alpha-1}-s^{\alpha-1}\right\rvert+\left\lvert t^\alpha-s^\alpha\right\rvert\right)\left\lVert\phi\right\rVert_H.
\end{align*}
\end{proof}
\begin{proposition}
For $T>0$, $S_t$ maps $H$ into $C\left((0,T];D(A)\right)$ with the estimate
\begin{equation}
\left\lVert S_t\phi\right\rVert_{D(A)}\lesssim\left(1+t^{-\alpha}\right)\left\lVert\phi\right\rVert_H,\quad t>0,
\label{lemmalemma linearlinear new2 equation1}
\end{equation}
and into $C\left([0,T];H\right)$ with the estimate
\begin{equation}
\left\lVert S_t\phi\right\rVert_H\lesssim\left\lVert\phi\right\rVert_H,\quad t\geq0.
\label{lemmalemma linearlinear new2 equation2}
\end{equation}
\label{lemma linear new2}
\end{proposition}
\begin{proof}[Proof]
The proof of $S_t$ maps $H$ into $C\left((0,T];H\right)$ and $C\left((0,T];D(A)\right)$ is left to Proposition $\ref{proposition linear new approach bilinear1}$ and the claim that $S_t$ is continuous at $t=0$ in the norm of $H$ can be proved by Lebesgue's dominated theorem. It suffices to prove $(\ref{lemmalemma linearlinear new2 equation1})$ and $(\ref{lemmalemma linearlinear new2 equation2})$. On one hand,
\begin{align*}
\left\lVert S_t^l\phi\right\rVert_{D(A)}&=\left\lVert S_t^l\phi\right\rVert_H+\left\lVert AS_t^l\phi\right\rVert_H\\
&=\left\lVert\chi_ta(t,\xi)U^{-1}\phi\right\rVert_{L^2(\Omega)}+\left\lVert a(\xi)\chi_ta(t,\xi)U^{-1}\phi\right\rVert_{L^2(\Omega)}\\
&\lesssim\left\lVert U^{-1}\phi\right\rVert_{L^2(\Omega)}+\left\lVert a(\xi)\chi_tU^{-1}\phi\right\rVert_{L^2(\Omega)}\\
&\lesssim\left\lVert U^{-1}\phi\right\rVert_{L^2(\Omega)}+t^{-\alpha}\left\lVert U^{-1}\phi\right\rVert_{L^2(\Omega)}\\
&=\left(1+t^{-\alpha}\right)\left\lVert\phi\right\rVert_H.
\end{align*}
On the other hand, it follows from Lemma $\ref{lemma linear new approach1}$ that
\[ \left\lVert S_t^h\phi\right\rVert_H\lesssim t^{-\alpha}\left\lVert\mathbf{A}_t^{-1}\phi\right\rVert_H+\left\lVert R_t^S\phi\right\rVert_H\lesssim\left\lVert\phi\right\rVert_H, \]
and
\[ \left\lVert AS_t^h\phi\right\rVert_H\lesssim t^{-\alpha}\left\lVert A\mathbf{A}_t^{-1}\phi\right\rVert_H+\left\lVert AR_t^S\phi\right\rVert_H\lesssim t^{-\alpha}\left\lVert\phi\right\rVert_H, \]
which implies that
\[ \left\lVert S_t^h\phi\right\rVert_{D(A)}\lesssim\left(1+t^{-\alpha}\right)\left\lVert\phi\right\rVert_H. \]
Combining above we can prove $(\ref{lemmalemma linearlinear new2 equation1})$ and $(\ref{lemmalemma linearlinear new2 equation2})$.
\end{proof}
\begin{proposition}
Let $\phi\in H$. For any $t,s>0$, we have
\begin{equation}
\left\lVert S_t\phi-S_s\phi\right\rVert_H\lesssim\left(1+\left(t\wedge s\right)^{-1}\right)\lvert t-s\rvert\left\lVert\phi\right\rVert_H
\label{proposition linear new approach bilinear1 equation1}
\end{equation}
and
\begin{equation}
\left\lVert AS_t\phi-AS_s\phi\right\rVert_H\lesssim\left(\left\lvert t^{-\alpha}-s^{-\alpha}\right\rvert+\left\lvert t^{1-\alpha}-s^{1-\alpha}\right\rvert\right)\left\lVert\phi\right\rVert_H.
\label{proposition linear new approach bilinear1 equation2}
\end{equation}
If moreover $\phi\in D(A)$, we have
\begin{equation}
\left\lVert S_t\phi-S_s\phi\right\rVert_H\lesssim\left(\left\lvert t^\alpha-s^\alpha\right\rvert+\left\lvert t^{\alpha+1}-s^{\alpha+1}\right\rvert\right)\left\lVert\phi\right\rVert_{D(A)}.
\label{proposition linear new approach bilinear1 equation3}
\end{equation}
\label{proposition linear new approach bilinear1}
\end{proposition}
\begin{proof}[Proof]
According to Lemma $\ref{lemma linear new approach bilinear1}$ and Lemma $\ref{lemma linear new approach bilinear2}$, it's sufficient to prove
\begin{gather}
\left\lVert S_t^l\phi-S_s^l\phi\right\rVert_H\lesssim\left(1+\left(t\wedge s\right)^{-1}\right)\lvert t-s\rvert\left\lVert\phi\right\rVert_H,\label{proposition linear new approach bilinear1 proof equation1}\\
\left\lVert S_t^l\phi-S_s^l\phi\right\rVert_H\lesssim\left(\left\lvert t^\alpha-s^\alpha\right\rvert+\left\lvert t^{\alpha+1}-s^{\alpha+1}\right\rvert\right)\left\lVert\phi\right\rVert_{D(A)},\label{proposition linear new approach bilinear1 proof equation3}
\end{gather}
and
\begin{equation}
\left\lVert AS_t^l\phi-AS_s^l\phi\right\rVert_H\lesssim\left(\left\lvert t^{-\alpha}-s^{-\alpha}\right\rvert+\left\lvert t^{1-\alpha}-s^{1-\alpha}\right\rvert\right)\left\lVert\phi\right\rVert_H.
\label{proposition linear new approach bilinear1 proof equation2}
\end{equation}
Since $\left\lvert a(\xi)b(t,\xi)\chi_t\right\rvert\lesssim t^{-1}$, $\left\lvert a(\xi)b(t,\xi)\chi_t\right\rvert\lesssim t^{\alpha-1}\left\lvert a(\xi)\right\rvert$, $\left\lvert a(t,\xi)\frac{d}{dt}\chi_t\right\rvert\lesssim1$ and $\left\lvert a(t,\xi)\frac{d}{dt}\chi_t\right\rvert\lesssim t^\alpha\left\lvert a(\xi)\right\rvert$, we have
\begin{align*}
\left\lvert\chi_ta(t,\xi)-\chi_sa(s,\xi)\right\rvert&=\left\lvert\int_s^t\frac{d}{d\tau}\left(\chi_\tau a(\tau,\xi)\right)d\tau\right\rvert\\
&=\left\lvert\int_s^tia(\xi)b(\tau,\xi)\chi_\tau+a(\tau,\xi)\frac{d}{d\tau}\chi_\tau d\tau\right\rvert\\
&\lesssim\left\lvert\int_s^t1+\tau^{-1}d\tau\right\rvert\\
&\leq\left(1+\left(t\wedge s\right)^{-1}\right)\lvert t-s\rvert,
\end{align*}
and
\begin{align*}
\left\lvert\chi_ta(t,\xi)-\chi_sa(s,\xi)\right\rvert&=\left\lvert\int_s^tia(\xi)b(\tau,\xi)\chi_\tau+a(\tau,\xi)\frac{d}{d\tau}\chi_\tau d\tau\right\rvert\\
&\lesssim\left\lvert\int_s^t\tau^{\alpha-1}\left\lvert a(\xi)\right\rvert+\tau^\alpha\left\lvert a(\xi)\right\rvert d\tau\right\rvert\\
&\lesssim\left(\left\lvert t^\alpha-s^\alpha\right\rvert+\left\lvert t^{\alpha+1}-s^{\alpha+1}\right\rvert\right)\left\lvert a(\xi)\right\rvert
\end{align*}
Using the fact that $\left\lvert a(t,\xi)\chi_t\right\rvert\lesssim t^{-\alpha}\left\lvert a(\xi)\right\rvert^{-1}$ and $\left\lvert a(\xi)b(t,\xi)\chi_t\right\rvert\lesssim t^{-\alpha-1}\left\lvert a(\xi)\right\rvert^{-1}$ we can obtain
\begin{align*}
\left\lvert\chi_ta(t,\xi)-\chi_sa(s,\xi)\right\rvert&=\left\lvert\int_s^tia(\xi)b(\tau,\xi)\chi_\tau+a(\tau,\xi)\frac{d}{d\tau}\chi_\tau d\tau\right\rvert\\
&\lesssim\left\lvert\int_s^t\tau^{-\alpha-1}\left\lvert a(\xi)\right\rvert^{-1}+\tau^{-\alpha}\left\lvert a(\xi)\right\rvert^{-1}d\tau\right\rvert\\
&\lesssim\left(\left\lvert t^{-\alpha}-s^{-\alpha}\right\rvert+\left\lvert t^{1-\alpha}-s^{1-\alpha}\right\rvert\right)\left\lvert a(\xi)\right\rvert^{-1}.
\end{align*}
Then $(\ref{proposition linear new approach bilinear1 proof equation1})$ follows from
\begin{align*}
\left\lVert S_t^l\phi-S_s^l\phi\right\rVert_H&=\left\lVert\chi_ta(t,\xi)U^{-1}\phi-\chi_sa(s,\xi)U^{-1}\phi\right\rVert_{L^2(\Omega)}\\
&\lesssim\left(1+\left(t\wedge s\right)^{-1}\right)\lvert t-s\rvert\left\lVert\phi\right\rVert_H,
\end{align*}
$(\ref{proposition linear new approach bilinear1 equation3})$ follows from
\begin{align*}
\left\lVert S_t^l\phi-S_s^l\phi\right\rVert_H&=\left\lVert\chi_ta(t,\xi)U^{-1}\phi-\chi_sa(s,\xi)U^{-1}\phi\right\rVert_{L^2(\Omega)}\\
&\lesssim\left(\left\lvert t^\alpha-s^\alpha\right\rvert+\left\lvert t^{\alpha+1}-s^{\alpha+1}\right\rvert\right)\left\lVert\phi\right\rVert_{D(A)},
\end{align*}
and $(\ref{proposition linear new approach bilinear1 proof equation2})$ follows from
\begin{align*}
\left\lVert AS_t^l\phi-AS_s^l\phi\right\rVert_H&=\left\lVert a(\xi)\chi_ta(t,\xi)U^{-1}\phi-a(\xi)\chi_sa(s,\xi)U^{-1}\phi\right\rVert_{L^2(\Omega)}\\
&\lesssim\left(\left\lvert t^{-\alpha}-s^{-\alpha}\right\rvert+\left\lvert t^{1-\alpha}-s^{1-\alpha}\right\rvert\right)\left\lVert\phi\right\rVert_H.
\end{align*}
\end{proof}
\begin{proposition}
For every $t>0$, if $\phi\in H$, then $g_{1-\alpha}(t)*(S_t\phi-\phi)$ is differentiable and
\begin{equation}
\mathbf{D}_t^\alpha(S_t\phi-\phi)=iAS_t\phi.
\label{lemma linear new4 equation1}
\end{equation}
\label{lemma linear new4}
\end{proposition}
\begin{proof}[Proof]
Let $\psi(t)=g_{1-\alpha}(t)*(S_t\phi-\phi)$. Since
\begin{align*}
\psi(t)&=\frac{1}{\Gamma(1-\alpha)}\int_0^t\left(t-\tau\right)^{-\alpha}\left(S_\tau\phi-\phi\right)d\tau\\
&=U\left(\frac{1}{\Gamma(1-\alpha)}\int_0^t\left(t-\tau\right)^{-\alpha}\left(a(\tau,\xi)U^{-1}\phi-U^{-1}\phi\right)d\tau\right)\\
&=U\left(\sum\limits_{k=1}^\infty\frac{i^ka(\xi)^kt^{\alpha(k-1)+1}}{\Gamma(\alpha(k-1)+2)}U^{-1}\phi\right),
\end{align*}
then we have
\begin{align*}
\lim\limits_{h\to0}\frac{\psi(t+h)-\psi(t)}{h}&=\lim\limits_{h\to0}U\left(\sum\limits_{k=1}^\infty\frac{i^ka(\xi)^k}{\Gamma(\alpha(k-1)+2)}\frac{\left(t+h\right)^{\alpha(k-1)+1}-t^{\alpha(k-1)+1}}{h}U^{-1}\phi\right)\\
&=U\left(\sum\limits_{k=1}^\infty\frac{i^ka(\xi)^k}{\Gamma(\alpha(k-1)+2)}\lim\limits_{h\to0}\frac{\left(t+h\right)^{\alpha(k-1)+1}-t^{\alpha(k-1)+1}}{h}U^{-1}\phi\right)\\
&=iU\left(a(\xi)a(t,\xi)U^{-1}\phi\right).
\end{align*}
Proposition $\ref{lemma linear new2}$ shows that $S_t\phi\in D(A)$ for every $t>0$, which implies that $g_{1-\alpha}(t)*(S_t\phi-\phi)$ is differentiable and
\[ \frac{d}{dt}\left(g_{1-\alpha}(t)*(S_t\phi-\phi)\right)=\mathbf{D}_t^\alpha\left(S_t\phi-\phi\right)=iAS_t\varphi \]
\end{proof}
\begin{proposition}
For $T>0$, $G$ maps $L^\infty\left((0,T);H\right)$ into $C\left([0,T];H\right)$ with the estimate
\begin{equation}
\left\lVert Gv\right\rVert_{L_T^\infty H}\lesssim T^\alpha\left\lVert v\right\rVert_{L_T^\infty H}.
\label{lemma linear new3 equation1}
\end{equation}
\label{lemma linear new3}
\end{proposition}
\begin{proof}[Proof]
The proof of continuity for $t>0$ is left to Proposition $\ref{8.16proposition linear new approach bilinear2}$ and the continuity at $t=0$ can be proved by $(\ref{lemma linear new3 proof equation3})$. We just prove $(\ref{lemma linear new3 equation1})$ here. On one hand,
\begin{equation}
\left\lVert G^lv(t)\right\rVert_H\lesssim\int_0^t\left(t-\tau\right)^{\alpha-1}\left\lVert v(\tau)\right\rVert_Hd\tau,
\label{lemma linear new3 proof equation1}
\end{equation}
it follows that
\[ \left\lVert G^lv\right\rVert_{L_T^\infty H}\lesssim T^\alpha\left\lVert v\right\rVert_{L_T^\infty H}. \]
On the other hand, according to Lemma $\ref{lemma linear new approach3}$ and Lemma $\ref{lemma linear new approach4}$, we obtain
\begin{equation}
\left\lVert G^hv(t)\right\rVert_H\lesssim\int_0^t\left(t-\tau\right)^{\alpha-1}\left\lVert v(\tau)\right\rVert_Hd\tau,
\label{lemma linear new3 proof equation2}
\end{equation}
which implies that
\[ \left\lVert G^hv\right\rVert_{L_T^\infty H}\lesssim T^\alpha\left\lVert v\right\rVert_{L_T^\infty H}. \]
Then $(\ref{lemma linear new3 equation1})$ can be proved.
\end{proof}
\begin{remark}
From $(\ref{lemma linear new3 proof equation1})$ and $(\ref{lemma linear new3 proof equation2})$ we also have
\begin{equation}
\left\lVert Gv(t)\right\rVert_H\lesssim\int_0^t\left(t-\tau\right)^{\alpha-1}\left\lVert v(\tau)\right\rVert_Hd\tau.
\label{lemma linear new3 proof equation3}
\end{equation}
In particular, for any $0<T_1<T_2$, if $v\equiv0$ on $[0,T_1]$ and $v\in L^\infty\left((0,T_2);H\right)$, then by $(\ref{lemma linear new3 proof equation3})$ we have
\[ \left\lVert Gv(t)\right\rVert_H\lesssim\int_{T_1}^t\left(t-\tau\right)^{\alpha-1}\left\lVert v(\tau)\right\rVert_Hd\tau \]
and hence
\begin{equation}
\left\lVert Gv\right\rVert_{L_{(T_1,T_2)}^\infty H}\lesssim\left(T_2-T_1\right)^\alpha\left\lVert v\right\rVert_{L_{T_2}^\infty H}.
\label{lemma linear new3 proof equation4}
\end{equation}
\end{remark}
\begin{proposition}
For $T>0$ and $0<t,s<T$, let $v\in L^\infty\left((0,T);H\right)$. We have
\begin{equation}
\left\lVert Gv(t)-Gv(s)\right\rVert_H\lesssim\left(\left\lvert t-s\right\rvert^\alpha+\left\lvert t^{\alpha+1}-s^{\alpha+1}\right\rvert\right)\left\lVert v\right\rVert_{L_T^\infty H}.
\label{8.16proposition linear new approach bilinear2 equation1}
\end{equation}
\label{8.16proposition linear new approach bilinear2}
\end{proposition}
\begin{proof}[Proof]
Without loss of generality we assume $t>s$. Let
\begin{align*}
&I_1=\frac{1}{\Gamma(-\alpha)}\int_0^s\left(\left(t-\tau\right)^{-\alpha-1}\mathbf{A}_{t-\tau}^{-2}v(\tau)-\left(s-\tau\right)^{-\alpha-1}\mathbf{A}_{s-\tau}^{-2}v(\tau)\right)d\tau,\\
&I_2=\frac{1}{\Gamma(-\alpha)}\int_s^t\left(t-\tau\right)^{-\alpha-1}\mathbf{A}_{t-\tau}^{-2}v(\tau)d\tau,\\
&I_3=\int_0^s\left(R_{t-\tau}^Pv(\tau)-R_{s-\tau}^Pv(\tau)\right)d\tau,\\
&I_4=\int_s^tR_{t-\tau}^Pv(\tau)d\tau,
\end{align*}
then $G^hv(t)-G^hv(s)=I_1+I_2+I_3+I_4$. According to Lemma $\ref{8.16lemma linear new approach bilinear3}$ and Lemma $\ref{8.16lemma linear new approach bilinear4}$, it follows that
\begin{align*}
\left\lVert I_1\right\rVert_H&\lesssim\int_0^s\left(\left(s-\tau\right)^{\alpha-1}-\left(t-\tau\right)^{\alpha-1}+\left(t-\tau\right)^\alpha-\left(s-\tau\right)^\alpha\right)\left\lVert v(\tau)\right\rVert_Hd\tau\\
&\lesssim\left(\left(t-s\right)^\alpha+t^{\alpha+1}-s^{\alpha+1}\right)\left\lVert v\right\rVert_{L_T^\infty H}
\end{align*}
and also
\[ \left\lVert I_3\right\rVert_H\lesssim\left(\left(t-s\right)^\alpha+t^{\alpha+1}-s^{\alpha+1}\right)\left\lVert v\right\rVert_{L_T^\infty H}. \]
Similarly it follows from Lemma $\ref{lemma linear new approach3}$ and Lemma $\ref{lemma linear new approach4}$ that
\[ \left\lVert I_2\right\rVert_H\lesssim\int_s^t\left(t-\tau\right)^{\alpha-1}\left\lVert v(\tau)\right\rVert_Hd\tau\lesssim\left(t-s\right)^\alpha\left\lVert v\right\rVert_{L_T^\infty H} \]
and also
\[ \left\lVert I_4\right\rVert_H\lesssim\left(t-s\right)^\alpha\left\lVert v\right\rVert_{L_T^\infty H}. \]
Then there holds
\begin{equation}
\left\lVert G^hv(t)-G^hv(s)\right\rVert_H\lesssim\left(\left(t-s\right)^\alpha+t^{\alpha+1}-s^{\alpha+1}\right)\left\lVert v\right\rVert_{L_T^\infty H}.
\label{8.16proposition linear new approach bilinear2 proof equation1}
\end{equation}
On the other hand, since
\begin{align*}
\left\lvert\chi_tb(t,\xi)-\chi_sb(s,\xi)\right\rvert&=\left\lvert\int_s^t\frac{d}{d\tau}\left(\chi_\tau b(\tau,\xi)\right)d\tau\right\rvert\\
&=\left\lvert\int_s^tb(\tau,\xi)\frac{d}{d\tau}\chi_\tau+\chi_\tau\tau^{\alpha-2}E_{\alpha,\alpha-1}\left(ia(\xi)\tau^\alpha\right)d\tau\right\rvert\\
&\lesssim\int_s^t\tau^{\alpha-1}+\tau^{\alpha-2}d\tau\\
&\lesssim t^\alpha-s^\alpha+s^{\alpha-1}-t^{\alpha-1},
\end{align*}
it follows that
\begin{align*}
\left\lVert P_t^l\phi-P_s^l\phi\right\rVert_H&=\left\lVert\chi_tb(t,\xi)U^{-1}\phi-\chi_sb(s,\xi)U^{-1}\phi\right\rVert_{L^2(\Omega)}\\
&\lesssim\left(t^\alpha-s^\alpha+s^{\alpha-1}-t^{\alpha-1}\right)\left\lVert\phi\right\rVert_H,
\end{align*}
which implies that
\[ \left\lVert I_5\right\rVert_H\lesssim\left(\left(t-s\right)^\alpha+t^{\alpha+1}-s^{\alpha+1}\right)\left\lVert v\right\rVert_{L_T^\infty H} \]
by letting
\begin{align*}
&I_5=\int_0^s\left(P_{t-\tau}^lv(\tau)-P_{s-\tau}^lv(\tau)\right)d\tau,\\
&I_6=\int_s^tP_{t-\tau}^lv(\tau)d\tau.
\end{align*}
There holds $G^lv(t)-G^lv(s)=I_5+I_6$. It's easy to verify that
\[ \left\lVert I_6\right\rVert_H\lesssim\left(t-s\right)^\alpha\left\lVert v\right\rVert_{L_T^\infty H}. \]
Then we have
\begin{equation}
\left\lVert G^lv(t)-G^lv(s)\right\rVert_H\lesssim\left(\left(t-s\right)^\alpha+t^{\alpha+1}-s^{\alpha+1}\right)\left\lVert v\right\rVert_{L_T^\infty H}.
\label{8.16proposition linear new approach bilinear2 proof equation2}
\end{equation}
Combining $(\ref{8.16proposition linear new approach bilinear2 proof equation1})$ and $(\ref{8.16proposition linear new approach bilinear2 proof equation2})$ we can obtain the result.
\end{proof}
\begin{proposition}
If $Gv(t)\in D(A)$ for $t>0$, then $g_{1-\alpha}*Gv$ is differentiable for $t>0$ and
\begin{equation}
\mathbf{D}_t^\alpha Gv(t)=iAGv(t)+v(t).
\label{lemma linear new6 equation1}
\end{equation}
\label{lemma linear new6}
\end{proposition}
\begin{proof}[Proof]
Let $\Phi(t)=g_{1-\alpha}*Gv$. Since
\begin{align*}
\Phi(t)&=\frac{1}{\Gamma(1-\alpha)}\int_0^t\int_0^\tau\left(t-\tau\right)^{-\alpha}P_{\tau-s}v(s)dsd\tau\\
&=U\left(\frac{1}{\Gamma(1-\alpha)}\int_0^t\int_0^\tau\left(t-\tau\right)^{-\alpha}b(\tau-s,\xi)U^{-1}v(s)dsd\tau\right)\\
&=U\left(\frac{1}{\Gamma(1-\alpha)}\int_s^t\int_0^t\left(t-\tau\right)^{-\alpha}b(\tau-s,\xi)U^{-1}v(s)dsd\tau\right)\\
&=U\left(\int_0^ta(t-s,\xi)U^{-1}v(s)ds\right),
\end{align*}
we obtain
\begin{align*}
&\lim\limits_{h\to0}\frac{\Phi(t+h)-\Phi(t)}{h}\\
&=U\left(\int_0^t\lim\limits_{h\to0}\frac{a(t+h-s,\xi)U^{-1}v(s)-a(t-s,\xi)U^{-1}v(s)}{h}ds\right)\\
&+U\left(\lim\limits_{h\to0}\frac{1}{h}\int_t^{t+h}a(t+h-s,\xi)U^{-1}v(s)ds\right)\\
&=iU\left(a(\xi)\int_0^tb(t-s,\xi)U^{-1}v(s)ds\right)+v(t).
\end{align*}
Since $Gv(t)\in D(A)$, we can deduce $\Phi(t)$ is differentiable and
\[ \mathbf{D}_t^\alpha Gv(t)=iAGv(t)+v(t). \]
\end{proof}
\begin{theorem}
For $T>0$, let $x\in H$ and $F\in L^\infty\left((0,T);H\right)$. The linear $(\ref{nonlinear Schrodinger equation Hilbert})$ has a unique mild solution $u$ on $[0,T]$ with the estimate
\begin{equation}
\left\lVert u(t)\right\rVert_H\lesssim\left\lVert x\right\rVert_H+T^\alpha\left\lVert F\right\rVert_{L_T^\infty H}.
\label{mild solution inhomogeneous H equation1}
\end{equation}
\label{mild solution inhomogeneous H}
\end{theorem}
\begin{proof}[Proof]
It's just a direct consequence of Proposition $\ref{lemma linear new2}$ and Proposition $\ref{lemma linear new3}$.
\end{proof}
\begin{theorem}
For $T>0$, let $F\in L^\infty\left((0,T);H\right)\cap C\left((0,T];H\right)$. The linear $(\ref{nonlinear Schrodinger equation Hilbert})$ has a unique classical solution $u$ on $[0,T]$ for every $x\in H$ if and only if $GF\in C\left((0,T];D(A)\right)$.
\label{linear classical solution theorem}
\end{theorem}
\begin{proof}[Proof]
If $u$ is a classical solution of the linear $(\ref{nonlinear Schrodinger equation Hilbert})$, then $GF=u-S_tx\in C\left((0,T];D(A)\right)$ by applying Proposition $\ref{lemma linear new2}$. If $GF\in C\left((0,T];D(A)\right)$, we can complete the proof applying Proposition $\ref{lemma linear new6}$ with the assumption that $F\in C\left((0,T];H\right)$.
\end{proof}
\begin{theorem}
For $T>0$, let $F\in C\left([0,T];D(A)\right)$. The linear $(\ref{nonlinear Schrodinger equation Hilbert})$ has a unique strict solution $u$ on $[0,T]$ for every $x\in D(A)$.
\label{linear strict solution theorem}
\end{theorem}
\begin{proof}[Proof]
The proof is clear noting that $G$ maps $C\left([0,T];D(A)\right)$ into $C\left([0,T];D(A)\right)$ by Proposition $\ref{lemma linear new3}$.
\end{proof}
\section{Proof of Theorem $\ref{unique mild solution X}$}
\label{9.17Proof of Theorem unique mild solution X}
We start with $(\ref{unique mild solution X theorem equation1})$ which also implies that the strict solution of $(\ref{nonlinear Schrodinger equation Hilbert})$ is unique. To this end, we state the following Gronwall type inequality which you can find in Henry\cite{Geometric-Theory-of-Semilinear-Parabolic-Equations} and Yagi\cite{Abstract-Parabolic-Evolution-Equations-and-their-Applications}.
\begin{lemma}
Let $0\leq a\in C\left([0,T];\mathbb{R}\right)$ be an increasing function, let $b>0$ be a constant and $\alpha>0$ be an exponent. If $u\in C\left([0,T];\mathbb{R}\right)$ satisfies the integral inequality
\[ u(t)\leq a(t)+b\int_0^t\left(t-s\right)^{\alpha-1}u(s)ds,\quad 0\leq t\leq T, \]
on this inteval, then
\[ u(t)\leq a(t)E_{\alpha,1}\left(b\Gamma(\alpha)t^\alpha\right),\quad 0\leq t\leq T. \]
\label{Yagislemma}
\end{lemma}
\begin{theorem}
Let the assumptions in Theorem $\ref{unique mild solution X}$ hold and $u(t),v(t)$ be the strict solutions of $(\ref{nonlinear Schrodinger equation Hilbert})$ with the initial data $x,y$ respectively, then $(\ref{unique mild solution X theorem equation1})$ holds.
\label{uniqueness}
\end{theorem}
\begin{proof}[Proof]
Let $R=\left\lVert u\right\rVert_{L_T^\infty D(A)}\vee\left\lVert v\right\rVert_{L_T^\infty D(A)}$. By the representation of the strict solution that $u(t)-v(t)=S_t(x-y)+iG(F(u)-F(v))$ and Proposition $\ref{lemma linear new2}$ and $(\ref{lemma linear new3 proof equation3})$ it follows that
\[ \left\lVert u(t)-v(t)\right\rVert_{D(A)}\lesssim_{\rho(x),\rho(y),R}\left\lVert x-y\right\rVert_{D(A)}+\int_0^t\left(t-\tau\right)^{\alpha-1}\left\lVert u(\tau)-v(\tau)\right\rVert_{D(A)}d\tau. \]
Then $(\ref{unique mild solution X theorem equation1})$ follows from Lemma $\ref{Yagislemma}$.
\end{proof}
\begin{lemma}
For $T>0$, $GF$ maps a ball with radius $R$ in $L^\infty\left((0,T);D(A)\right)$ into $C\left([0,T];D(A)\right)$ boundedly with the estimate
\[ \left\lVert GF(u)-GF(v)\right\rVert_{L_T^\infty D(A)}\lesssim_{\rho(u(0)),\rho(v(0)),R}T^\alpha\left\lVert u-v\right\rVert_{L_T^\infty D(A)}. \]
\label{unique mild solution X proof lemma1}
\end{lemma}
\begin{proof}[Proof]
By hypotheses on $F$ we obtain $F$ maps a ball with radius $R$ in $L^\infty\left((0,T);D(A)\right)$ into $L^\infty\left((0,T);D(A)\right)$ with the estimate
\[ \left\lVert F(u)-F(v)\right\rVert_{L_T^\infty D(A)}\lesssim_{\rho(u(0)),\rho(v(0)),R}\left\lVert u-v\right\rVert_{L_T^\infty D(A)}. \]
Then the result can be followed by Proposition $\ref{lemma linear new3}$.
\end{proof}
\begin{lemma}
For $T>0$, let $x\in D(A)$ such that $\left\lvert x\right\rvert_1<\infty$. If $u\in C\left([0,T];D(A)\right)$ satisfies $u(t)=S_tx+iGF(u)$, then $u$ is a strict solution of $(\ref{nonlinear Schrodinger equation Hilbert})$ on $[0,T]$.
\label{inhomogeneous mild solution strict}
\end{lemma}
\begin{proof}[Proof]
Let $H(t)=F(u)$ and $R=\left\lVert u\right\rVert_{L_T^\infty D(A)}$. It's easy to verify that $H(t)\in D(A)$ and for any $t_0\in[0,T]$,
\[ \left\lVert H(t)-H(t_0)\right\rVert_{D(A)}\lesssim_{\rho(x),R}\left\lVert u(t)-u(t_0)\right\rVert_{D(A)}\to0,\quad t\to t_0. \]
It shows that $H\in C\left([0,T];D(A)\right)$. Applying Theorem $\ref{linear strict solution theorem}$ there exsits a unique strict solution $v$ on $[0,T]$ to the following equation
\[ iD_t^\alpha v(t)+Av(t)+H(t)=0,\quad v(0)=x. \]
Then it follows that $v(t)=S_tx+iGH(t)=S_tx+iGF(u)=u(t)$ which completes the proof.
\end{proof}
\begin{proof}[\textbf{Proof of Theorem} $\mathbf{\ref{unique mild solution X}}$]
Let $R=\left\lVert x\right\rVert_{D(A)}$ and set
\[ X_R=\left\{u\in L^\infty\left((0,T);D(A)\right):u(0)=x,\quad\left\lVert u\right\rVert_{L_T^\infty D(A)}\lesssim R\right\} \]
with metric
\[ d(u,v)=\left\lVert u-v\right\rVert_{L_T^\infty D(A)}. \]
Define $Ku=S_tx+iGF(u)$. It follows from Proposition $\ref{lemma linear new2}$ and Lemma $\ref{unique mild solution X proof lemma1}$ that for any $u\in X_R$,
\[ \left\lVert Ku\right\rVert_{L_T^\infty D(A)}\lesssim\left\lVert x\right\rVert_{D(A)}+C_{\rho(x),R}T^\alpha\left\lVert u\right\rVert_{L_T^\infty D(A)}\lesssim R+C_{\rho(x),R}T^\alpha. \]
Then we can choose $T$ small enough such that $K$ maps $X_R$ into $X_R$. For any $u,v\in X_R$, we have, by Lemma $\ref{unique mild solution X proof lemma1}$,
\[ d\left(Ku,Kv\right)=\left\lVert GF(u)-GF(v)\right\rVert_{L_T^\infty D(A)}\lesssim_{\rho(x),R}T^\alpha\left\lVert u-v\right\rVert_{L_T^\infty D(A)}=T^\alpha d(u,v). \]
We can choose $T$ small enough such that $K$ is a contraction on $X_R$ which implies that $K$ has a unique fixed point $u\in X_R$ and hence $u=Ku\in C\left([0,T];D(A)\right)$ by Lemma $\ref{unique mild solution X proof lemma1}$. Combining with Lemma $\ref{inhomogeneous mild solution strict}$ we can obtain the local existence and uniqueness of the strict solution.
\end{proof}
\section{Proof of Theorem $\ref{continuation and blow up alternative}$}
\label{Proof of Theorem continuation and blow up alternative}
\begin{lemma}[continuation]
Let $u$ be the strict solution of $(\ref{nonlinear Schrodinger equation Hilbert})$ on $[0,T]$ under the assumptions in Theorem \ref{unique mild solution X}. Then $u$ can be extended to the interval $[0,T^*]$ for some $T^*>T$ uniquely and the extended function is the strict solution of $(\ref{nonlinear Schrodinger equation Hilbert})$ on $[0,T^*]$.
\label{extended function strict solution}
\end{lemma}
\begin{proof}[Proof]
Due to Lemma $\ref{inhomogeneous mild solution strict}$, we only need to prove $u$ can be extended to $v\in C\left([0,T^*];D(A)\right)$ satisfying $v(t)=S_tx+iGF(v)$ on $[0,T^*]$. Let $Kv=S_tx+iGF(v)$, $R=\left\lVert u\right\rVert_{L_T^\infty D(A)}$ and set
\[ E_R=\left\{v\in L^\infty\left((0,T^*);D(A)\right):\begin{array}{cc}
v\equiv u\;on\;[0,T]\\
\left\lVert v-u(T)\right\rVert_{L_{(T,T^*)}^\infty D(A)}\lesssim R\;on\;[T,T^*]
\end{array}\right\} \]
with the metric
\[ d(v,w)=\left\lVert v-w\right\rVert_{L_{T^*}^\infty D(A)}. \]
We first claim that $K$ maps $E_R$ into itself. Indeed, for any $v\in E_R$, clearly we have, by Proposition $\ref{lemma linear new2}$ and Lemma $\ref{unique mild solution X proof lemma1}$, that $Kv\in L^\infty\left((0,T^*);D(A)\right)$. And $Kv\equiv Ku=u$ on $[0,T]$ follows from that $u$ is a strict solution on $[0,T]$. Now on $[T,T^*]$, according to Proposition $\ref{proposition linear new approach bilinear1}$ and Proposition $\ref{8.16proposition linear new approach bilinear2}$, we can choose $T^*$ and $T$ to be close enough such that
\begin{align*}
&\left\lVert Kv-u(T)\right\rVert_{L_{(T,T^*)}^\infty D(A)}\\
&\lesssim_{\rho(x),R}\left(1+T^{-1}\right)\left(T^*-T\right)\left\lVert x\right\rVert_{D(A)}+\left(\left(T^*-T\right)^\alpha+\left(T^{*\alpha+1}-T^{\alpha+1}\right)\right)\left\lVert v\right\rVert_{L_{T^*}^\infty D(A)}\\
&\lesssim_{\rho(x),R}\left(1+T^{-1}\right)\left(T^*-T\right)R+\left(\left(T^*-T\right)^{\alpha+1}+T\left(T^*-T\right)^\alpha+T^\alpha\left(T^*-T\right)\right)R\\
&\lesssim R.
\end{align*}
Then it follows that $K$ maps $E_R$ into itself. Similarly, for any $v,w\in E_R$, it follows from $(\ref{lemma linear new3 proof equation4})$ that
\begin{align*}
d\left(Kv,Kw\right)&=\left\lVert Kv-Kw\right\rVert_{L_{(T,T^*)}^\infty D(A)}=\left\lVert GF(v)-GF(w)\right\rVert_{L_{(T,T^*)}^\infty D(A)}\\
&\lesssim\left(T^*-T\right)^\alpha\left\lVert F(v)-F(w)\right\rVert_{L_{T^*}^\infty D(A)}\lesssim_{\rho(x),R}\left(T^*-T\right)^\alpha d(v,w).
\end{align*}
Then we can choose $T^*$ and $T$ to be close enough such that $K$ is a contraction on $E_R$ which completes the proof of the result.
\end{proof}
\begin{proof}[\textbf{Proof of Theorem} $\mathbf{\ref{continuation and blow up alternative}}$]
Let
\[ T_{\max}=\sup\left\{T\in(0,\infty):\exists!\;strict\;solution\;to\;(\ref{nonlinear Schrodinger equation Hilbert})\;on\;[0,T]\right\}. \]
Lemma $\ref{extended function strict solution}$ shows the existence of $T_{\max}$ and that $u\in C\left([0,T_{\max});D(A)\right)$. It's clear that $0<T_{\max}\leq\infty$. Suppose that $T_{\max}<\infty$ but $\left\lVert u(t)\right\rVert_{D(A)}\lesssim R$ on $[0,T_{\max}]$. When $t\to T_{\max}$, assuming $t\in\left[T_{\max}-\delta,T_{\max}\right]$, by Proposition $\ref{proposition linear new approach bilinear1}$ and Proposition $\ref{8.16proposition linear new approach bilinear2}$, we have
\begin{align*}
&\left\lVert u(t)-u(T_{\max})\right\rVert_{D(A)}\\
&\lesssim_{\rho(x),R}\left(1+\left(T_{\max}-\delta\right)^{-1}\right)\left(T_{\max}-t\right)\left\lVert x\right\rVert_{D(A)}+\left(\left(T_{\max}-t\right)^\alpha+\left(T_{\max}^{\alpha+1}-t^{\alpha+1}\right)\right)\left\lVert u\right\rVert_{L_{T_{\max}}^\infty D(A)}\\
&\lesssim_{\rho(x),R}\left(1+\left(T_{\max}-\delta\right)^{-1}\right)\left(T_{\max}-t\right)+\left(\left(T_{\max}-t\right)^\alpha+\left(T_{\max}^{\alpha+1}-t^{\alpha+1}\right)\right)\\
&\to0,\quad t\to T_{\max}.
\end{align*}
It follows that $u\in C\left([0,T_{\max}];D(A)\right)$. But by Lemma $\ref{extended function strict solution}$, $u$ can be extended to the interval $[0,T^*]$ for some $T^*>T_{\max}$ which contradicts the definition of $T_{\max}$. Then $T_{\max}<\infty$ implies that $\lim\limits_{t\to T_{\max}}\left\lVert u(t)\right\rVert_{D(A)}=\infty$.
\end{proof}
\section{Proof of Theorem $\ref{unique global solution X}$}
\label{Proof of Theorem unique global solution X}
\begin{lemma}
Let $w:[0,\infty)\to[0,\infty)$ be a continuous, nondecreasing, nonnegative function which is not always $0$ and $u(t)$ be a continuous, nonnegative function on $[0,T]$ satisfying
\[ u(t)\lesssim1+\int_0^t\left(t-\tau\right)^{\alpha-1}w\left(u(\tau)\right)d\tau \]
where $\alpha\in(0,1)$. Then
\[ \int_1^{u(t)}\frac{\sigma^{\frac{1}{\alpha}+\varepsilon-1}}{w(\sigma)^{\frac{1}{\alpha}+\varepsilon}}d\sigma\lesssim_\varepsilon e^{\left(\frac{1}{\alpha}+\varepsilon\right)T}\left(1-e^{-\left(\frac{1}{\alpha}+\varepsilon\right)t}\right). \]
\label{9.9proof of Theorem unique global solution X lemma}
\end{lemma}
\begin{proof}[Proof]
The proof is similar to Theorem 2 in \cite{Integral-Inequalities-and-Global-Solutions-of-Semilinear-Evolution-Equations} and we omit it.
\end{proof}
\begin{proof}[\textbf{Proof of Theorem} $\mathbf{\ref{unique global solution X}}$]
It suffices to prove that $\left\lVert u(t)\right\rVert_{D(A)}$ is bounded on every interval $[0,T]$. We assume that there exists a $T<\infty$ such that $\lim\limits_{t\uparrow T}\left\lVert u(t)\right\rVert_{D(A)}=\infty$. Then for such $T$, by Proposition $\ref{lemma linear new2}$, Assumption $\ref{AssumptionB'}$ and $(\ref{lemma linear new3 proof equation3})$, we obtain
\[ \left\lVert u(t)\right\rVert_{D(A)}\lesssim_{\left\lVert x\right\rVert_{D(A)},\varrho(x)}1+\int_0^t\left(t-\tau\right)^{\alpha-1}w\left(\left\lVert u(\tau)\right\rVert_{D(A)}\right)d\tau. \]
Then Lemma $\ref{9.9proof of Theorem unique global solution X lemma}$ shows that
\[ \int_1^{\left\lVert u(t)\right\rVert_{D(A)}}\frac{\sigma^{\frac{1}{\alpha}+\varepsilon-1}}{w(\sigma)^{\frac{1}{\alpha}+\varepsilon}}d\sigma\lesssim_{\varepsilon,\left\lVert x\right\rVert_{D(A)},\varrho(x)}e^{\left(\frac{1}{\alpha}+\varepsilon\right)T}. \]
Letting $t\uparrow T$ on both sides of the inequality we have
\[ \int_1^\infty\frac{\sigma^{\frac{1}{\alpha}+\varepsilon-1}}{w(\sigma)^{\frac{1}{\alpha}+\varepsilon}}d\sigma\lesssim_{\varepsilon,\left\lVert x\right\rVert_{D(A)},\varrho(x)}e^{\left(\frac{1}{\alpha}+\varepsilon\right)T}. \]
This leads to a contradiction since $w\in C_\infty[0,\infty)$. Then the proof is complete.
\end{proof}
\section{Application. The well-posedness of the fractional dispersive equation}
\label{9.16Application. The well-posedness of the fractional dispersive equation}
In this section, we will consider the well-posedness of $(\ref{9.1 very general dispersive equation motivated})$. The hypotheses on $F$ are:
\begin{assumption}
$F(0)=0$ and $\left\lVert F(u)-F(v)\right\rVert_{H^s(\mathbb{R}^n)}\lesssim_{\rho(u(0)),\rho(v(0)),R}\left\lVert u-v\right\rVert_{H^s(\mathbb{R}^n)}$ a.e. on $I\subset [0,\infty)$, where $R$ is the essential upper bound of $\left\lVert u(t)\right\rVert_{H^s(\mathbb{R}^n)}$ and $\left\lVert v(t)\right\rVert_{H^s(\mathbb{R}^n)}$ on $I$.
\label{AssumptionA}
\end{assumption}

\begin{assumption}
There exists a $w\in C_\infty[0,\infty)$ such that $\left\lVert F(u)\right\rVert_{H^s(\mathbb{R}^n)}\lesssim_{\varrho(u(0))}w\left(\left\lVert u(t)\right\rVert_{H^s(\mathbb{R}^n)}\right)$.
\label{AssumptionB}
\end{assumption}
Here $\rho(\cdot)$ and $\varrho(\cdot)$ are norms.

We will prove the following results.
\begin{theorem}[local well-posedness]
Let $\alpha\in(0,1)$, $m\geq\frac{n}{2}$, $s\geq m$, $q\in L^2(\mathbb{R}^n)$ and $V\in L^\infty(\mathbb{R}^n)$. $P(D)$ is defined as $P(D)u=\mathscr{F}^{-1}\left(P(\xi)\mathscr{F}u\right)$ where the assumption of $P(\xi)$ is $P(\xi)\in C\left(\mathbb{R}^n;\mathbb{R}\right)$ and $\left\lvert P(\xi)\right\rvert\sim\left\lvert\xi\right\rvert^m$ when $\lvert\xi\rvert\to\infty$ and Assumption $\ref{AssumptionA}$ holds. If $u_0\in H^s(\mathbb{R}^n)$ satisfying $\left\lVert u_0\right\rVert_{H^s(\mathbb{R}^n)}\vee\rho(u_0)<\infty$, there exists a positive number $T$ which depends only on $\left\lVert u_0\right\rVert_{H^s(\mathbb{R}^n)}$ and $\rho(u_0)$ such that $(\ref{9.1 very general dispersive equation motivated equivalent})$ admits a unique solution on $[0,T]$ in the class
\[ u\in C\left([0,T];H^s(\mathbb{R}^n)\right),\quad\mathbf{D}_t^\alpha(u-u_0)\in C\left([0,T];L^2(\mathbb{R}^n)\right). \]
In addition, $u$ can be extended to the maximal interval $[0,T_{\max})$ such that 
\[ u\in C\left([0,T_{\max});H^s(\mathbb{R}^n)\right),\quad\mathbf{D}_t^\alpha(u-u_0)\in C\left([0,T_{\max});L^2(\mathbb{R}^n)\right) \]
and $T_{\max}<\infty$ implies $\lim\limits_{t\uparrow T_{max}}\left\lVert u(t)\right\rVert_{H^s(\mathbb{R}^n)}=\infty$.
\label{9.1 very general dispersive equation motivated equivalent result}
\end{theorem}
\begin{theorem}[global well-posedness]
Let the assumptions in Theorem $\ref{9.1 very general dispersive equation motivated equivalent result}$ and Assumption $\ref{AssumptionB}$ hold. If $u_0\in H^s(\mathbb{R}^n)$ satisfying $\left\lVert u_0\right\rVert_{H^s(\mathbb{R}^n)}\vee\rho(u_0)\vee\varrho(u_0)<\infty$, $(\ref{9.1 very general dispersive equation motivated equivalent})$ admits a unique solution on $[0,\infty)$ in the class
\[ u\in C\left([0,\infty);H^s(\mathbb{R}^n)\right),\quad\mathbf{D}_t^\alpha(u-u_0)\in C\left([0,\infty);L^2(\mathbb{R}^n)\right). \]
\label{9.1 very general dispersive equation motivated equivalent result global}
That is, the solution in Theorem $\ref{9.1 very general dispersive equation motivated equivalent result}$ is global.
\end{theorem}
To deal with these theorems, it suffices to consider the following equivalent equation that
\begin{equation}
\begin{cases}
iD_t^\alpha u+P(D)u+qu+Vu+F(u)=0,\quad&x\in\mathbb{R}^n,\;t>0\\
u(0,x)=u_0(x),\quad&x\in\mathbb{R}^n
\end{cases}
\label{9.1 very general dispersive equation motivated equivalent}
\end{equation}
where the assumptions of $P(D), q, V, F$ are the same as Theorem $\ref{9.1 very general dispersive equation motivated equivalent result}$ and Theorem $\ref{9.1 very general dispersive equation motivated equivalent result global}$.

Define the following operators:
\begin{align*}
&Hu=\mathscr{F}^{-1}\left(P(\xi)\mathscr{F}u\right),\quad D(H)=H^s(\mathbb{R}^n),\;s\geq m,\\
&Q_1u=qu,\quad D(Q_1)=\left\{u\in L^2(\mathbb{R}^n):qu\in L^2(\mathbb{R}^n)\right\},\\
&Q_2u=Vu,\quad D(Q_2)=\left\{u\in L^2(\mathbb{R}^n):Vu\in L^2(\mathbb{R}^n)\right\},
\end{align*}
and $Tu=Q_1u+Q_2u$, $Au=Hu+Tu$. $Q_1,Q_2$ is called the the maximal multiplication operators by $q,V$ respectively and hence they are selfadjoint operator in $L^2(\mathbb{R}^n)$ (see \cite{Perturbation-Theory-for-Linear-Operators}). We claim that $D(H)\subset\left\{u\in L^2(\mathbb{R}^n):P(D)u\in L^2(\mathbb{R}^n)\right\}$. Indeed, by the assumption $s\geq m$, we obtain $\frac{P(\xi)}{\left(1+\left\lvert\xi\right\rvert^2\right)^{\frac{s}{2}}}\in L^\infty(\mathbb{R}^n)$ and hence
\[ \left\lVert P(D)u\right\rVert_{L^2(\mathbb{R}^n)}\sim\left\lVert\frac{P(\xi)}{\left(1+\left\lvert\xi\right\rvert^2\right)^{\frac{s}{2}}}\left(1+\left\lvert\xi\right\rvert^2\right)^{\frac{s}{2}}\mathscr{F}u\right\rVert_{L^2(\mathbb{R}^n)}<\infty. \]
Also it's easy to see that $H$ is a selfadjoint operator in $L^2(\mathbb{R}^n)$.
\begin{proposition}
$A$ is a selfadjoint operator in $L^2(\mathbb{R}^n)$ with $D(A)=D(H)$.
\label{9.1 introductioin selfadjoint motivated}
\end{proposition}
\begin{proof}[Proof]
We choose $\gamma$ large enough and then the asymptotic behavior of $P(\xi)$ shows that
\begin{align*}
\int_{\mathbb{R}^n}\frac{1}{P(\xi)^2+\gamma^2}d\xi&=\int_{\lvert\xi\rvert\leq\gamma^{\frac{1}{m}}}\frac{1}{P(\xi)^2+\gamma^2}d\xi+\int_{\lvert\xi\rvert>\gamma^{\frac{1}{m}}}\frac{1}{P(\xi)^2+\gamma^2}d\xi\\
&\lesssim\gamma^{\frac{n}{m}-2}+\int_{\lvert\xi\rvert>\gamma^{\frac{1}{m}}}\frac{1}{\left\lvert\xi\right\rvert^{2m}+\gamma^2}d\xi\\
&\sim\gamma^{\frac{n}{m}-2}.
\end{align*}
It follows that
\begin{align*}
\left(\int_{\mathbb{R}^n}\left\lvert\mathscr{F}u\right\rvert d\xi\right)^2&=\left(\int_{\mathbb{R}^n}\frac{1}{P(\xi)+\gamma}(P(\xi)+\gamma)\left\lvert\mathscr{F}u\right\rvert d\xi\right)^2\\
&\leq\int_{\mathbb{R}^n}\frac{1}{\left(P(\xi)+\gamma\right)^2}d\xi\int_{\mathbb{R}^n}\left(P(\xi)+\gamma\right)^2\left\lvert\mathscr{F}u\right\rvert^2d\xi\\
&\lesssim\int_{\mathbb{R}^n}\frac{1}{P(\xi)^2+\gamma^2}d\xi\left(\left\lVert P(D)u\right\rVert_{L^2(\mathbb{R}^n)}^2+\gamma^2\left\lVert u\right\rVert_{L^2(\mathbb{R}^n)}^2\right)\\
&\lesssim\gamma^{\frac{n}{m}-2}\left\lVert P(D)u\right\rVert_{L^2(\mathbb{R}^n)}^2+\gamma^{\frac{n}{m}}\left\lVert u\right\rVert_{L^2(\mathbb{R}^n)}^2\\
&=\gamma^{\frac{n}{m}-2}\left\lVert Hu\right\rVert_{L^2(\mathbb{R}^n)}^2+\gamma^{\frac{n}{m}}\left\lVert u\right\rVert_{L^2(\mathbb{R}^n)}^2
\end{align*}
and hence
\[ \left\lVert u\right\rVert_{L^\infty(\mathbb{R}^n)}\lesssim\left\lVert\mathscr{F}u\right\rVert_{L^1(\mathbb{R}^n)}\lesssim\gamma^{\frac{n}{2m}-1}\left\lVert Hu\right\rVert_{L^2(\mathbb{R}^n)}+\gamma^{\frac{n}{2m}}\left\lVert u\right\rVert_{L^2(\mathbb{R}^n)}. \]
Since
\begin{align*}
&\begin{aligned}
\left\lVert Q_1u\right\rVert_{L^2(\mathbb{R}^n)}&=\left\lVert qu\right\rVert_{L^2(\mathbb{R}^n)}\leq\left\lVert q\right\rVert_{L^2(\mathbb{R}^n)}\left\lVert u\right\rVert_{L^\infty(\mathbb{R}^n)}\\
&\lesssim\gamma^{\frac{n}{2m}-1}\left\lVert q\right\rVert_{L^2(\mathbb{R}^n)}\left\lVert Hu\right\rVert_{L^2(\mathbb{R}^n)}+\gamma^{\frac{n}{2m}}\left\lVert q\right\rVert_{L^2(\mathbb{R}^n)}\left\lVert u\right\rVert_{L^2(\mathbb{R}^n)},
\end{aligned}\\
&\left\lVert Q_2u\right\rVert_{L^2(\mathbb{R}^n)}=\left\lVert Vu\right\rVert_{L^2(\mathbb{R}^n)}\leq\left\lVert V\right\rVert_{L^\infty(\mathbb{R}^n)}\left\lVert u\right\rVert_{L^2(\mathbb{R}^n)},
\end{align*}
we have
\begin{equation}
\left\lVert Tu\right\rVert_{L^2(\mathbb{R}^n)}\lesssim\left(\gamma^{\frac{n}{2m}}\left\lVert q\right\rVert_{L^2(\mathbb{R}^n)}+\left\lVert V\right\rVert_{L^\infty(\mathbb{R}^n)}\right)\left\lVert u\right\rVert_{L^2(\mathbb{R}^n)}+\gamma^{\frac{n}{2m}-1}\left\lVert q\right\rVert_{L^2(\mathbb{R}^n)}\left\lVert Hu\right\rVert_{L^2(\mathbb{R}^n)}.
\label{9.1 introductioin selfadjoint motivated proof equation1}
\end{equation}
Thus we can choose $\gamma$ large enough such that $T$ is $H$-bounded with $H$-bound smaller than $1$ and then $A$ is a selfadjoint operator in $L^2(\mathbb{R}^n)$ by Theorem $\ref{9.1 perturbation selfadjoint operator}$.

Now it remains to prove $D(A)=D(H)$. $(\ref{9.1 introductioin selfadjoint motivated proof equation1})$ shows
\begin{align*}
\left\lVert Au\right\rVert_{L^2(\mathbb{R}^n)}&\leq\left\lVert Hu\right\rVert_{L^2(\mathbb{R}^n)}+\left\lVert Tu\right\rVert_{L^2(\mathbb{R}^n)}\\
&\lesssim\left(\gamma^{\frac{n}{2m}}\left\lVert q\right\rVert_{L^2(\mathbb{R}^n)}+\left\lVert V\right\rVert_{L^\infty(\mathbb{R}^n)}\right)\left\lVert u\right\rVert_{L^2(\mathbb{R}^n)}+\left(1+\gamma^{\frac{n}{2m}-1}\left\lVert q\right\rVert_{L^2(\mathbb{R}^n)}\right)\left\lVert Hu\right\rVert_{L^2(\mathbb{R}^n)}
\end{align*}
and hence $D(H)\subset D(A)$. On the other hand, since
\begin{align*}
&\left\lVert Hu\right\rVert_{L^2(\mathbb{R}^n)}\\
&\leq\left\lVert Au\right\rVert_{L^2(\mathbb{R}^n)}+\left\lVert Tu\right\rVert_{L^2(\mathbb{R}^n)}\\
&\leq\left\lVert Au\right\rVert_{L^2(\mathbb{R}^n)}+C\left(\gamma^{\frac{n}{2m}}\left\lVert q\right\rVert_{L^2(\mathbb{R}^n)}+\left\lVert V\right\rVert_{L^\infty(\mathbb{R}^n)}\right)\left\lVert u\right\rVert_{L^2(\mathbb{R}^n)}+C\gamma^{\frac{n}{2m}-1}\left\lVert q\right\rVert_{L^2(\mathbb{R}^n)}\left\lVert Hu\right\rVert_{L^2(\mathbb{R}^n)},
\end{align*}
we can choose $\gamma$ large enough such that $C\gamma^{\frac{n}{2m}-1}\left\lVert q\right\rVert_{L^2(\mathbb{R}^n)}<1$ and then
\begin{align*}
&\left\lVert Hu\right\rVert_{L^2(\mathbb{R}^n)}\\
&\leq\left(1-C\gamma^{\frac{n}{2m}-1}\left\lVert q\right\rVert_{L^2(\mathbb{R}^n)}\right)^{-1}\left\lVert Au\right\rVert_{L^2(\mathbb{R}^n)}\\
&+C\left(\gamma^{\frac{n}{2m}}\left\lVert q\right\rVert_{L^2(\mathbb{R}^n)}+\left\lVert V\right\rVert_{L^\infty(\mathbb{R}^n)}\right)\left(1-C\gamma^{\frac{n}{2m}-1}\left\lVert q\right\rVert_{L^2(\mathbb{R}^n)}\right)^{-1}\left\lVert u\right\rVert_{L^2(\mathbb{R}^n)}
\end{align*}
which implies that $D(A)\subset D(H)$. Then the proof is complete.
\end{proof}
By the above arguments and applying Theorem $\ref{unique mild solution X}$, Theorem $\ref{continuation and blow up alternative}$ and Theorem $\ref{unique global solution X}$ we can obtain Theorem $\ref{9.1 very general dispersive equation motivated equivalent result}$ and Theorem $\ref{9.1 very general dispersive equation motivated equivalent result global}$.
\begin{appendices}
\section{On the fractional integral and fractional derivatives}
\label{9.1Fractionl integral and derivative definition appendix}
Here we give a brief introduction to the fractional integral and fractional derivatives. We only consider $0<\alpha<1$ without particular comment. Let
\[ g_\alpha(t):=\begin{cases}
\frac{t^{\alpha-1}}{\Gamma(\alpha)},\quad&t>0\\
0,\quad&t\leq0
\end{cases}.\]
We can now define the Rimann-Liouville fractional integral ($I_t^\alpha$), the Riemann-Liouville fractional derivarive ($\mathbf{D}_t^\alpha$) and the Caputo derivative ($D_t^\alpha$) by
\[ I_t^\alpha u(t)=g_\alpha*u,\quad\mathbf{D}_t^\alpha u(t)=\frac{d}{dt}\left(g_{1-\alpha}*u\right),\quad D_t^\alpha u(t)=g_{1-\alpha}*u'(t), \]
and the relationship between the Riemann-Liouville derivative and the Caputo derivative is given by
\[ D_t^\alpha u(t)=\mathbf{D}_t^\alpha\left(u(t)-u(0)\right). \]
\section{On the Mittag-Leffler functions}
\label{9.17On the Mittag-Leffler functions}
Here we give a brief introduction to the Mittag-Leffler functions which are the fundamental functions in fractional differential equations.
\begin{definition}[\cite{Theory-and-Applications-of-Fractional-Differential-Equations-Chapter1-Preliminaries}]
Let $\alpha,\beta,z\in\mathbb{C}$ and $\R\alpha>0$, we define the Mittag-Leffler function by
\[ E_{\alpha,\beta}(z):=\sum\limits_{k=0}^\infty\frac{z^k}{\Gamma(\alpha k+\beta)}. \]
\end{definition}
We state the asymptotic expansion and the derivative of the Mittag-Leffler function as follows.
\begin{theorem}[\cite{Fractional-differential-equations-an-introduction-to-fractional-derivatives-fractional-differential-equations-to-methods-of-their-solution-and-some-of-their-applications}]
If $0<\alpha<2$, $\beta$ is an arbitrary complex number and $\mu$ is an arbitrary real number such that
\[ \frac{\pi\alpha}{2}<\mu<\pi\wedge\pi\alpha, \]
then for an arbitrary interger $p\geq1$ the following expansion holds
\[ E_{\alpha,\beta}(z)=-\sum\limits_{k=1}^p\frac{z^{-k}}{\Gamma(\beta-\alpha k)}+O\left(\left\lvert z\right\rvert^{-1-p}\right),\quad\lvert z\rvert\to\infty,\;\mu\leq\lvert\arg z\rvert\leq\pi. \]
\label{Mittag-Leffler function asymptotic expansion}
\end{theorem}
\begin{theorem}
Let $0<\alpha<1$ and $\lambda\in\mathbb{C}$, then
\begin{gather*}
\frac{d}{dt}E_{\alpha,1}\left(\lambda t^\alpha\right)=\lambda t^{\alpha-1}E_{\alpha,\alpha}\left(\lambda t^\alpha\right),\\
\frac{d}{dt}\left(t^{\alpha-1}E_{\alpha,\alpha}\left(\lambda t^\alpha\right)\right)=t^{\alpha-2}E_{\alpha,\alpha-1}\left(\lambda t^\alpha\right).
\end{gather*}
\label{Derivative Mittag Leffler}
\end{theorem}
\section{The perturbation and the spectral theorem of the selfadjoint operators}
\label{9.17The perturbation and the spectral theorem of the selfadjoint operators}
\subsection{The perturbation of the selfadjoint operators}
Recall that an operator $A\in\mathscr{C}(X,Y)$ is relatively bounded with respect to $T\in\mathscr{C}(X,Y)$ (or $T$-bounded) if $D(A)\supset D(T)$ and
\[ \left\lVert Au\right\rVert_Y\leq a\left\lVert u\right\rVert_X+b\left\lVert Tu\right\rVert_Y, \]
where $X,Y$ are Banach spaces and $b$ is called $T$-bound.
\begin{theorem}[\cite{Perturbation-Theory-for-Linear-Operators}]
Let $H$ be selfadjoint. If $T$ is symmetric and $H$-bounded with $H$-bound smaller than $1$, then $H+T$ is selfadjoint.
\label{9.1 perturbation selfadjoint operator}
\end{theorem}
\subsection{The spectral theorem of the selfadjoint operators}
\label{9.5appendix The spectral theorem of the selfadjoint operator}
Let $H$ be a separable Hilbert space and $A$ be a selfadjoint operator. We denote $H$ by $D(A^0)$ and endow $D(A)$ with the graph norm $\left\lVert x\right\rVert_{D(A)}:=\left\lVert x\right\rVert_H+\left\lVert Ax\right\rVert_H$. Define $D\left(A^n\right),n\geq2$ by
\[ D\left(A^n\right):=\left\{x\in D\left(A^{n-1}\right):A^{n-1}x\in D(A)\right\}, \]
with the graph norm
\[ \left\lVert x\right\rVert_{D\left(A^n\right)}:=\left\lVert x\right\rVert_H+\left\lVert A^nx\right\rVert_H. \]
Note that $A^n$ is selfadjoint and $D(A^n)$ is a Banach space and also a Hilbert space. By induction, it's easy to check the following equivalence form,
\[ D(A^n)=\begin{cases}
H,\quad &n=0\\
D(A),\quad &n=1\\
\left\{x\in D(A):Ax\in D(A),\cdots,A^{n-1}x\in D(A)\right\},\quad &n\geq2
\end{cases}. \]
\begin{lemma}
Let $H$ be a separable Hilbert space, $A$ be a self-adjoint operator on $H$ with domain $D(A)$. Then there exists a measure space $\left(\Omega,\mu\right)$ with $\mu$ a finite measure, a unitary operator $U:L^2(\Omega)\to H$ and a real-valued function $a(\xi)$ on $\Omega$ which is finite a.e. such that,
\begin{enumerate}[label=(\roman*)]
\item $\psi\in D\left(A^n\right)\Longleftrightarrow\bigcup\limits_{k=0}^n\left\{a(\xi)^kU^{-1}\psi\right\}\subset L^2(\Omega)$,\quad $n\geq0$,
\item $U\varphi\in D(A^n)\Longrightarrow A^nU\varphi=U\left(a(\xi)^n\varphi\right)$,\quad $n\geq0$.
\end{enumerate}
Moreover, if $A$ is injective, we have, with $D\left(A^{-n}\right)=R\left(A^n\right)$,
\begin{enumerate}[label=(\alph*)]
\item $\psi\in D\left(A^{-n}\right)\Longleftrightarrow\bigcup\limits_{k=0}^n\left\{a(\xi)^{-k}U^{-1}\psi\right\}\subset L^2(\Omega)$,\quad $n\geq0$,
\item $U\varphi\in D(A^{-n})\Longrightarrow A^{-n}U\varphi=U\left(a(\xi)^{-n}\varphi\right)$,\quad $n\geq0$.
\end{enumerate}
In addition, the measure space $(\Omega,\mu)$ and the function $a(\xi)$ can be chosen such that $a\in L^p(\Omega)$ for all $p$ with $1\leq p<\infty.$
\label{Local solution lemma}
\end{lemma}
\begin{proof}[Proof]
\textbf{Proof of} $\mathbf{(i)}$ \textbf{and} $\mathbf{(ii)}$. The case $n=0$ is trival. The proof of the case $n=1$ you can see \cite{Methods-of-Modern-Mathematical-Physics-VIII-Unbounded-Operators}. Assume that $(ii)$ is valid in the case $U\varphi\in D(A^{n-1})$. When $U\varphi\in D(A^n)$,
\[ A^nU\varphi=AU\left(a(\xi)^{n-1}\psi\right)=U\left(a(\xi)^n\varphi\right), \]
which completes the proof of $(ii)$. Similarly, Assume that $(i)$ is valid for the case $n-1$. For the case $n$, by the definition of $D(A^n)$, we have
\[ \psi\in D(A^n)\Longleftrightarrow\psi\in D(A^{n-1})\;and\;A^{n-1}\psi\in D(A). \]
On one hand, by assumption,
\[ \psi\in D(A^{n-1})\Longleftrightarrow\bigcup\limits_{k=0}^{n-1}\left\{a(\xi)^kU^{-1}\psi\right\}\subset L^2(\Omega), \]
on the other hand,
\[ A^{n-1}\psi\in D(A)\Longleftrightarrow a(\xi)U^{-1}\left(A^{n-1}\psi\right)\in L^2(\Omega)\Leftrightarrow a(\xi)^nU^{-1}\psi\in L^2(\Omega), \]
then we obtain
\[ \psi\in D\left(A^n\right)\Longleftrightarrow\bigcup\limits_{k=0}^n\left\{a(\xi)^kU^{-1}\psi\right\}\subset L^2(\Omega). \]
Then the proof of $(i)$ is complete.

\textbf{Proof of} $\mathbf{(a)}$ \textbf{and} $\mathbf{(b)}$. We just need to prove the case $n\geq1$. Recall that
\[ \psi\in D(A^{-n})\Longleftrightarrow\exists\varphi\in D(A^n)\;s.t.\;\psi=A^n\varphi. \]
By $(ii)$,
\[ U^{-1}\psi=U^{-1}\left(A^n\varphi\right)=a(\xi)^nU^{-1}\varphi, \]
that is
\[ a(\xi)^{-n}U^{-1}\psi=U^{-1}\varphi, \]
it follows that
\[ \psi\in D(A^{-n})\Longleftrightarrow\exists\varphi\in D(A^n)\;s.t.\;a(\xi)^{-n}U^{-1}\psi=U^{-1}\varphi. \]
But by $(i)$,
\[ \varphi\in D(A^n)\Longleftrightarrow\bigcup\limits_{k=0}^n\left\{a(\xi)^kU^{-1}\varphi\right\}\subset L^2(\Omega), \]
we can deduce that
\begin{align*}
\psi\in D(A^{-n})&\Longleftrightarrow\bigcup\limits_{k=0}^n\left\{a(\xi)^{-(n-k)}U^{-1}\psi\right\}\subset L^2(\Omega)\\
&\Longleftrightarrow\bigcup\limits_{k=0}^n\left\{a(\xi)^{-k}U^{-1}\psi\right\}\subset L^2(\Omega).
\end{align*}
Then we complete the proof of $(a)$ and $(b)$.
\end{proof}
Note that $a(\xi)\neq0$ a.e. on $\Omega$ which you can see the proof of Theorem VIII.4 in \cite{Methods-of-Modern-Mathematical-Physics-VIII-Unbounded-Operators}
\section{Some further results of the linear $(\ref{nonlinear Schrodinger equation Hilbert})$}
\label{Some further results of the linear}
\subsection{H\"older regularities}
\begin{proposition}
For $T>0$ and $0<t,s<T$, let $\frac{1}{2}<\alpha<1$ and $v\in L^q\left((0,T);H\right)$ where $\frac{1}{2\alpha-1}<q<\infty$. We have
\begin{equation}
\left\lVert Gv(t)-Gv(s)\right\rVert_H\lesssim\left(T^{\frac{1}{q'}+4\alpha-3}+T^{\frac{1}{q'}+2\alpha-1}+T^{\frac{1}{q'}-2(1-\alpha)}\right)\left\lvert t-s\right\rvert^{1-\alpha}\left\lVert v\right\rVert_{L_T^qH}.
\label{8.28 linear Holder regularity proposition1 equation1}
\end{equation}
\label{8.28 linear Holder regularity proposition1}
\end{proposition}
\begin{proof}[Proof]
Using the notations in Proposition $\ref{8.16proposition linear new approach bilinear2}$, we first have
\begin{align*}
\left\lVert I_1\right\rVert_H&\lesssim\int_0^s\left(\left(s-\tau\right)^{\alpha-1}-\left(t-\tau\right)^{\alpha-1}+\left(t-\tau\right)^\alpha-\left(s-\tau\right)^\alpha\right)\left\lVert v(\tau)\right\rVert_Hd\tau\\
&\leq\left(\int_0^s\left(\left(s-\tau\right)^{\alpha-1}-\left(t-\tau\right)^{\alpha-1}+\left(t-\tau\right)^\alpha-\left(s-\tau\right)^\alpha\right)^{q'}d\tau\right)^{\frac{1}{q'}}\left\lVert v\right\rVert_{L_T^qH}\\
&\lesssim\left(\left(t-s\right)^{1-\alpha}\left(\int_0^s\left(s-\tau\right)^{(\alpha-1)q'}\left(t-\tau\right)^{(\alpha-1)q'}d\tau\right)^{\frac{1}{q'}}+s^{\frac{1}{q'}}\left(t-s\right)^\alpha\right)\left\lVert v\right\rVert_{L_T^qH}\\
&\leq\left(\left(t-s\right)^{1-\alpha}\left(\int_0^s\left(s-\tau\right)^{2(\alpha-1)q'}d\tau\right)^{\frac{1}{q'}}+s^{\frac{1}{q'}}\left(t-s\right)^\alpha\right)\left\lVert v\right\rVert_{L_T^qH}\\
&\lesssim\left(s^{\frac{1}{q'}-2(1-\alpha)}\left(t-s\right)^{1-\alpha}+s^{\frac{1}{q'}}\left(t-s\right)^\alpha\right)\left\lVert v\right\rVert_{L_T^qH}\\
&\leq\left(T^{\frac{1}{q'}+4\alpha-3}+T^{\frac{1}{q'}+2\alpha-1}\right)\left(t-s\right)^{1-\alpha}\left\lVert v\right\rVert_{L_T^qH}
\end{align*}
and
\[ \left\lVert I_3\right\rVert_H\lesssim\left(T^{\frac{1}{q'}+4\alpha-3}+T^{\frac{1}{q'}+2\alpha-1}\right)\left(t-s\right)^{1-\alpha}\left\lVert v\right\rVert_{L_T^qH}. \]
Also we have
\begin{align*}
\left\lVert I_2\right\rVert_H&\lesssim\int_s^t\left(t-\tau\right)^{\alpha-1}\left\lVert v(\tau)\right\rVert_Hd\tau\\
&\leq\left(\int_s^t\left(t-\tau\right)^{(\alpha-1)q'}d\tau\right)^{\frac{1}{q'}}\left\lVert v\right\rVert_{L_T^qH}\\
&\lesssim\left(t-s\right)^{\frac{1}{q'}-(1-\alpha)}\left\lVert v\right\rVert_{L_T^qH}\\
&\leq T^{\frac{1}{q'}-2(1-\alpha)}\left(t-s\right)^{1-\alpha}\left\lVert v\right\rVert_{L_T^qH}
\end{align*}
and
\[ \left\lVert I_4\right\rVert_H\lesssim T^{\frac{1}{q'}-2(1-\alpha)}\left(t-s\right)^{1-\alpha}\left\lVert v\right\rVert_{L_T^qH}. \]
Then there holds
\[ \left\lVert G^hv(t)-G^hv(s)\right\rVert_H\lesssim\left(T^{\frac{1}{q'}+4\alpha-3}+T^{\frac{1}{q'}+2\alpha-1}+T^{\frac{1}{q'}-2(1-\alpha)}\right)\left(t-s\right)^{1-\alpha}\left\lVert v\right\rVert_{L_T^qH}. \]
On the other hand, since
\begin{align*}
\left\lVert I_5\right\rVert_H&\lesssim\int_0^s\left(\left(s-\tau\right)^{\alpha-1}-\left(t-\tau\right)^{\alpha-1}+\left(t-\tau\right)^\alpha-\left(s-\tau\right)^\alpha\right)\left\lVert v(\tau)\right\rVert_Hd\tau\\
&\lesssim\left(T^{\frac{1}{q'}+4\alpha-3}+T^{\frac{1}{q'}+2\alpha-1}\right)\left(t-s\right)^{1-\alpha}\left\lVert v\right\rVert_{L_T^qH}
\end{align*}
and
\begin{align*}
\left\lVert I_6\right\rVert_H&\lesssim\int_s^t\left(t-\tau\right)^{\alpha-1}\left\lVert v(\tau)\right\rVert_Hd\tau\\
&\lesssim T^{\frac{1}{q'}-2(1-\alpha)}\left(t-s\right)^{1-\alpha}\left\lVert v\right\rVert_{L_T^qH},
\end{align*}
it follows that
\[ \left\lVert G^lv(t)-G^lv(s)\right\rVert_H\lesssim\left(T^{\frac{1}{q'}+4\alpha-3}+T^{\frac{1}{q'}+2\alpha-1}+T^{\frac{1}{q'}-2(1-\alpha)}\right)\left(t-s\right)^{1-\alpha}\left\lVert v\right\rVert_{L_T^qH}. \]
Then $(\ref{8.28 linear Holder regularity proposition1 equation1})$ thus holds.
\end{proof}
\begin{theorem}
For $T>0$, let $x\in H$, $F\in L^\infty\left((0,T);H\right)$ and $u$ be the mild solution of the linear $(\ref{nonlinear Schrodinger equation Hilbert})$ on $[0,T]$, then $u\in C^\alpha\left([\delta,T];H\right)$ for every $0<\delta<T$ with the estimate
\begin{equation}
\left[u\right]_{C_{[\delta,T]}^\alpha H}\lesssim T^{1-\alpha}\left(1+\delta^{-1}\right)\left\lVert x\right\rVert_H+T\left\lVert F\right\rVert_{L_T^\infty H}.
\label{8.28 linear Holder regularity theorem1 equation1}
\end{equation}
If moreover $x\in D(A)$, $u\in C^\alpha\left([0,T];H\right)$ with the estimate
\begin{equation}
\left[u\right]_{C_T^\alpha H}\lesssim(1+T)\left(\left\lVert x\right\rVert_{D(A)}+\left\lVert F\right\rVert_{L_T^\infty H}\right).
\label{8.28 linear Holder regularity theorem1 equation2}
\end{equation}
\label{8.28 linear Holder regularity theorem1}
\end{theorem}
\begin{proof}[Proof]
By the representation of the mild solution that $u=S_tx+iGF(t)$ and Proposition $\ref{proposition linear new approach bilinear1}$, Proposition $\ref{8.16proposition linear new approach bilinear2}$ we can obtain
\begin{align*}
\left\lVert u(t)-u(s)\right\rVert_H&\lesssim\left(1+\left(t\wedge s\right)^{-1}\right)\left\lvert t-s\right\rvert\left\lVert x\right\rVert_H+\left(\left\lvert t-s\right\rvert^\alpha+\left\lvert t^{\alpha+1}-s^{\alpha+1}\right\rvert\right)\left\lVert F\right\rVert_{L_T^\infty H}\\
&\lesssim T^{1-\alpha}\left(1+\delta^{-1}\right)\left\lvert t-s\right\rvert^\alpha\left\lVert x\right\rVert_H+T\left\lvert t-s\right\rvert^\alpha\left\lVert F\right\rVert_{L_T^\infty H}
\end{align*}
which implies $(\ref{8.28 linear Holder regularity theorem1 equation1})$. And for $x\in D(A)$, we have
\begin{align*}
\left\lVert u(t)-u(s)\right\rVert_H&\lesssim\left(\left\lvert t-s\right\rvert^\alpha+\left\lvert t^{\alpha+1}-s^{\alpha+1}\right\rvert\right)\left(\left\lVert x\right\rVert_{D(A)}+\left\lVert F\right\rVert_{L_T^\infty H}\right)\\
&\lesssim(1+T)\left\lvert t-s\right\rvert^\alpha\left(\left\lVert x\right\rVert_{D(A)}+\left\lVert F\right\rVert_{L_T^\infty H}\right)
\end{align*}
which implies $(\ref{8.28 linear Holder regularity theorem1 equation2})$.
\end{proof}
\begin{theorem}
For $T>0$, let $x\in H$, $F\in L^q\left((0,T);H\right)$ where $\frac{1}{2\alpha-1}<q<\infty$ and $u$ be the mild solution the linear $(\ref{nonlinear Schrodinger equation Hilbert})$ on $[0,T]$, then $u\in C^{1-\alpha}\left([\delta,T];H\right)$ for any $0<\delta<T$ with the estimate
\begin{equation}
\left[u\right]_{C_{[\delta,T]}^{1-\alpha}H}\lesssim\left(1+\delta^{-1}\right)T^\alpha\left\lVert x\right\rVert_H+\left(T^{\frac{1}{q'}+4\alpha-3}+T^{\frac{1}{q'}+2\alpha-1}+T^{\frac{1}{q'}-2(1-\alpha)}\right)\left\lVert F\right\rVert_{L_T^qH}.
\label{8.29 linear Holder regularity theorem2 equation1}
\end{equation}
If moreover $x\in D(A)$, $u\in C^{1-\alpha}\left([0,T];H\right)$ with the estimate
\begin{equation}
\left[u\right]_{C_T^{1-\alpha}H}\lesssim\left(T^{2\alpha-1}+T^{2\alpha}\right)\left\lVert x\right\rVert_{D(A)}+\left(T^{\frac{1}{q'}+4\alpha-3}+T^{\frac{1}{q'}+2\alpha-1}+T^{\frac{1}{q'}-2(1-\alpha)}\right)\left\lVert F\right\rVert_{L_T^qH}.
\label{8.29 linear Holder regularity theorem2 equation2}
\end{equation}
\label{8.29 linear Holder regularity theorem2}
\end{theorem}
\begin{proof}[Proof]
To simplify, let $C=T^{\frac{1}{q'}+4\alpha-3}+T^{\frac{1}{q'}+2\alpha-1}+T^{\frac{1}{q'}-2(1-\alpha)}$. With the help of Proposition $\ref{proposition linear new approach bilinear1}$ and Proposition $\ref{8.28 linear Holder regularity proposition1}$ it follows that
\begin{align*}
\left\lVert u(t)-u(s)\right\rVert_H&\lesssim\left(1+\left(t\wedge s\right)^{-1}\right)\left\lvert t-s\right\rvert\left\lVert x\right\rVert_H+C\left\lvert t-s\right\rvert^{1-\alpha}\left\lVert F\right\rVert_{L_T^qH}\\
&\leq\left(1+\delta^{-1}\right)T^\alpha\left\lvert t-s\right\rvert^{1-\alpha}\left\lVert x\right\rVert_H+C\left\lvert t-s\right\rvert^{1-\alpha}\left\lVert F\right\rVert_{L_T^qH}
\end{align*}
which implies $(\ref{8.29 linear Holder regularity theorem2 equation1})$. And for $x\in D(A)$, we have
\begin{align*}
\left\lVert u(t)-u(s)\right\rVert_H&\lesssim\left(\left\lvert t-s\right\rvert^\alpha+\left\lvert t^{\alpha+1}-s^{\alpha+1}\right\rvert\right)\left\lVert x\right\rVert_{D(A)}+C\left\lvert t-s\right\rvert^{1-\alpha}\left\lVert F\right\rVert_{L_T^qH}\\
&\lesssim\left(T^{2\alpha-1}+T^{2\alpha}\right)\left\lvert t-s\right\rvert^{1-\alpha}\left\lVert x\right\rVert_{D(A)}+C\left\lvert t-s\right\rvert^{1-\alpha}\left\lVert F\right\rVert_{L_T^qH}
\end{align*}
which implies $(\ref{8.29 linear Holder regularity theorem2 equation2})$.
\end{proof}
\subsection{Asymptotic behaviors}
It's easy to show that if $x\in H$, $F\in L^\infty\left((0,\infty);H\right)$, then there is a mild solution $u\in C\left([0,\infty);H\right)$ of the linear $(\ref{nonlinear Schrodinger equation Hilbert})$ on $[0,\infty)$ satisfying $u(t)=S_tx+iGF(t)$.
\begin{theorem}
If $A$ is injective, let $x\in D(A^{-1})$, $F\in L^\infty\left((0,\infty);H\right)$ and $u$ be the mild solution of the linear $(\ref{nonlinear Schrodinger equation Hilbert})$ on $[0,\infty)$. If there exists $F_0\in D(A^{-1})$ such that
\[ \lim\limits_{t\to\infty}\int_0^t\left(t-\tau\right)^{\alpha-1}\left\lVert F(\tau)-F_0\right\rVert_Hd\tau=0, \]
then $u$ satisfies
\[ \lim\limits_{t\to\infty}u(t)=-A^{-1}F_0. \]
\label{8.29 Asymptotic behaviours theorem1}
\end{theorem}
\begin{proof}[Proof]
Thanks to Lemma $\ref{lemma linear new approach1}$, Lemma $\ref{lemma linear new approach2}$ and Proposition $\ref{lemma linear new2}$, we have
\[ \left\lVert S_tx\right\rVert_H\lesssim t^{-\alpha}\left\lVert x\right\rVert_{D(A^{-1})} \]
which implies that $\lim\limits_{t\to\infty}S_tx=0$. On the other hand, we can devide $GF(t)$ into two parts such that
\[ GF(t)=\int_0^tP_{t-\tau}\left(F(\tau)-F_0\right)d\tau+\int_0^tP_{t-\tau}F_0d\tau=:v_1(t)+v_2(t). \]
A straightforward computation leads to
\begin{align*}
v_2(t)&=\int_0^tP_{t-\tau}F_0d\tau=U\left(\int_0^tb(t-\tau,\xi)d\tau U^{-1}F_0\right)\\
&=U\left(\int_0^tia(\xi)^{-1}\frac{d}{d\tau}a(t-\tau,\xi)d\tau U^{-1}F_0\right)\\
&=iA^{-1}F_0-iA^{-1}S_tF_0.
\end{align*}
By Lebesgue's dominated theorem, there holds
\begin{align*}
\left\lVert iA^{-1}S_tF_0\right\rVert_H&\leq\left\lVert A^{-1}S_t^lF_0\right\rVert_H+\left\lVert A^{-1}S_t^hF_0\right\rVert_H\\
&\lesssim\left\lVert A^{-1}S_t^lF_0\right\rVert_H+\left\lVert t^{-\alpha}A^{-1}\mathbf{A}_t^{-1}F_0\right\rVert_H+\left\lVert A^{-1}R_t^SF_0\right\rVert_H\\
&=\left\lVert a(\xi)^{-1}\chi_ta(t,\xi)U^{-1}F_0\right\rVert_{L^2(\Omega)}+\left\lVert t^{-\alpha}a(\xi)^{-2}\chi_t^cU^{-1}F_0\right\rVert_{L^2(\Omega)}\\
&+\left\lVert t^{-2\alpha}a(\xi)^{-1}O\left(\left\lvert a(\xi)\right\rvert^{-2}\right)\chi_t^cU^{-1}F_0\right\rVert_{L^2(\Omega)}\\
&\to0,\quad t\to\infty.
\end{align*}
This implies that $\lim\limits_{t\to\infty}v_2(t)=iA^{-1}F_0$. By Assumption and $(\ref{lemma linear new3 proof equation3})$ we obtain
\[ \left\lVert v_1(t)\right\rVert_H\lesssim\int_0^t\left(t-\tau\right)^{\alpha-1}\left\lVert F(\tau)-F_0\right\rVert_Hd\tau\to0,\quad t\to\infty. \]
It follows that $\lim\limits_{t\to\infty}GF(t)=iA^{-1}F_0$ and hence the result holds.
\end{proof}
\begin{theorem}
Let $F\in L^\infty\left((0,\infty);H\right)$ and $x\in H$. If $u_\varepsilon(t)$ is the mild solution of
\begin{equation}
iD_t^\alpha u_\varepsilon(t)+\varepsilon Au_\varepsilon(t)+F(t)=0,\quad u_\varepsilon(0)=x,
\label{8.29 Asymptotic behaviours theorem2 equation1}
\end{equation}
on $[0,\infty)$, then
\begin{equation}
\lim\limits_{\varepsilon\to0}u_\varepsilon(t)=x+iI_t^\alpha F(t)
\label{8.29 Asymptotic behaviours theorem2 equation2}
\end{equation}
on $[0,\infty)$ pointwisely.
\label{8.29 Asymptotic behaviours theorem2}
\end{theorem}
\begin{proof}[Proof]
Clearly $u_\varepsilon(t)$ exists and satisfies
\[ u_\varepsilon(t)=S_{\varepsilon^{\frac{1}{\alpha}}t}x+i\varepsilon^{\frac{1}{\alpha}-1}\int_0^tP_{\varepsilon^{\frac{1}{\alpha}}(t-\tau)}F(\tau)d\tau. \]
Thanks to Lebesgue's dominated theorem, it follows that
\begin{align*}
\left\lVert S_{\varepsilon^{\frac{1}{\alpha}}t}x-x\right\rVert_H&\leq\left\lVert S_{\varepsilon^{\frac{1}{\alpha}}t}^lx-x\right\rVert_H+\left\lVert S_{\varepsilon^{\frac{1}{\alpha}}t}^hx\right\rVert_H\\
&\lesssim\left\lVert\chi_{\varepsilon^{\frac{1}{\alpha}}t}a\left(\varepsilon^{\frac{1}{\alpha}}t,\xi\right)U^{-1}x-U^{-1}x\right\rVert_{L^2(\Omega)}+\left\lVert\varepsilon^{-1}t^{-\alpha}a(\xi)^{-1}\chi_{\varepsilon^{\frac{1}{\alpha}}t}^cU^{-1}x\right\rVert_{L^2(\Omega)}\\
&+\left\lVert\varepsilon^{-2}t^{-2\alpha}O\left(\left\lvert a(\xi)\right\rvert^{-2}\right)\chi_{\varepsilon^{\frac{1}{\alpha}}t}^cU^{-1}x\right\rVert_{L^2(\Omega)}\\
&\to0,\quad \varepsilon\to0
\end{align*}
and hence $\lim\limits_{\varepsilon\to0}S_{\varepsilon^{\frac{1}{\alpha}}t}x=x$. On the other hand, also it follows from Lebesgue's dominated theorem that
\begin{align*}
&\left\lVert\varepsilon^{\frac{1}{\alpha}-1}\int_0^tP_{\varepsilon^{\frac{1}{\alpha}}(t-\tau)}F(\tau)d\tau-\frac{1}{\Gamma(\alpha)}\int_0^t\left(t-\tau\right)^{\alpha-1}F(\tau)d\tau\right\rVert_H\\
&\leq\left\lVert\varepsilon^{\frac{1}{\alpha}-1}\int_0^tP_{\varepsilon^{\frac{1}{\alpha}}(t-\tau)}^lF(\tau)d\tau-\frac{1}{\Gamma(\alpha)}\int_0^t\left(t-\tau\right)^{\alpha-1}F(\tau)d\tau\right\rVert_H+\left\lVert\varepsilon^{\frac{1}{\alpha}-1}\int_0^tP_{\varepsilon^{\frac{1}{\alpha}}(t-\tau)}^hF(\tau)d\tau\right\rVert_H\\
&\lesssim\left\lVert\varepsilon^{\frac{1}{\alpha}-1}\int_0^t\chi_{\varepsilon^{\frac{1}{\alpha}}(t-\tau)}b\left(\varepsilon^{\frac{1}{\alpha}}(t-\tau),\xi\right)U^{-1}F(\tau)d\tau-\frac{1}{\Gamma(\alpha)}\int_0^t\left(t-\tau\right)^{\alpha-1}U^{-1}F(\tau)d\tau\right\rVert_{L^2(\Omega)}\\
&+\left\Vert\varepsilon^{-2}\int_0^t\left(t-\tau\right)^{-\alpha-1}a(\xi)^{-2}\chi_{\varepsilon^{\frac{1}{\alpha}}(t-\tau)}^cU^{-1}F(\tau)d\tau\right\rVert_{L^2(\Omega)}\\
&+\left\Vert\varepsilon^{-3}\int_0^t\left(t-\tau\right)^{-2\alpha-1}\left\lvert a(\xi)\right\rvert^{-3}\chi_{\varepsilon^{\frac{1}{\alpha}}(t-\tau)}^cU^{-1}F(\tau)d\tau\right\rVert_{L^2(\Omega)}\\
&\to0,\quad\varepsilon\to0
\end{align*}
and hence $(\ref{8.29 Asymptotic behaviours theorem2 equation2})$ holds.
\end{proof}
\begin{theorem}
If $A$ is injective, let $0<\alpha<\frac{1}{\alpha}$, $F\in L^\infty\left((0,\infty);D(A^{-1})\right)$ be continuous and bounded on $(0,\infty)$ and $x\in H$. If $u_\varepsilon$ is the mild solution of
\begin{equation}
i\varepsilon D_t^\alpha u_\varepsilon(t)+Au_\varepsilon(t)+F(t)=0,\quad u_\varepsilon(0)=x,
\label{8.30 Asymptotic behaviours theorem3 equation1}
\end{equation}
on $[0,\infty)$, then
\begin{equation}
\lim\limits_{\varepsilon\to0}u_\varepsilon(t)=-A^{-1}F(t).
\label{8.30 Asymptotic behaviours theorem3 equation2}
\end{equation}
uniformly on $[\delta,T]$ for any $0<\delta<T$.
\label{8.30 Asymptotic behaviours theorem3}
\end{theorem}
\begin{proof}[Proof]
Clearly $u_\varepsilon(t)$ exists and satisfies
\[ u_\varepsilon(t)=S_{\varepsilon^{-\frac{1}{\alpha}}t}x+i\varepsilon^{-\frac{1}{\alpha}}\int_0^tP_{\varepsilon^{-\frac{1}{\alpha}}(t-\tau)}F(\tau)d\tau. \]
By Lebesgue's dominated theorem we can obtain
\begin{equation}
\begin{aligned}
\left\lVert S_{\varepsilon^{-\frac{1}{\alpha}}t}x\right\rVert_H&\leq\left\lVert S_{\varepsilon^{-\frac{1}{\alpha}}t}^lx\right\rVert_H+\left\lVert S_{\varepsilon^{-\frac{1}{\alpha}}t}^hx\right\rVert_H\\
&\lesssim\left\lVert\chi_{\varepsilon^{-\frac{1}{\alpha}}t}a\left(\varepsilon^{-\frac{1}{\alpha}}t,\xi\right)U^{-1}x\right\rVert_{L^2(\Omega)}+\left\lVert\varepsilon t^{-\alpha}a(\xi)^{-1}\chi_{\varepsilon^{-\frac{1}{\alpha}}t}^cU^{-1}x\right\rVert_{L^2(\Omega)}\\
&+\left\lVert\varepsilon^2t^{-2\alpha}\left\lvert a(\xi)\right\rvert^{-2}\chi_{\varepsilon^{-\frac{1}{\alpha}}t}^cU^{-1}x\right\rVert_{L^2(\Omega)}\\
&\to0,\quad\varepsilon\to0
\end{aligned}
\label{8.30 Asymptotic behaviours theorem3 proof equation1}
\end{equation}
and the limit is uniform on $[\delta,T]$. Dividing the second term into two parts such that
\[ \varepsilon^{-\frac{1}{\alpha}}\int_0^tP_{\varepsilon^{-\frac{1}{\alpha}}(t-\tau)}F(\tau)d\tau=v_{1\varepsilon}(t)+v_{2\varepsilon}(t), \]
where
\begin{align*}
&v_{1\varepsilon}(t)=\varepsilon^{-\frac{1}{\alpha}}\int_0^tP_{\varepsilon^{-\frac{1}{\alpha}}(t-\tau)}\left(F(\tau)-F(t)\right)d\tau,\\
&v_{2\varepsilon}(t)=\varepsilon^{-\frac{1}{\alpha}}\int_0^tP_{\varepsilon^{-\frac{1}{\alpha}}(t-\tau)}F(t)d\tau,
\end{align*}
a straightforward computation leads to
\[ v_{2\varepsilon}(t)=iA^{-1}F(t)-iA^{-1}S_{\varepsilon^{-\frac{1}{\alpha}}t}F(t). \]
A similar way as $(\ref{8.30 Asymptotic behaviours theorem3 proof equation1})$ we can prove 
\[ \left\lVert A^{-1}S_{\varepsilon^{-\frac{1}{\alpha}}t}F(t)\right\rVert_H\to0,\quad\varepsilon\to0 \]
uniformly on $[\delta,T]$ by the boundedness of $F(t)$ and hence $\lim\limits_{\varepsilon\to0}v_{2\varepsilon}(t)=iA^{-1}F(t)$ uniformly on $[\delta,T]$. On the other hand, we can choose $r$ large enough and $\varepsilon$ small enough such that
\begin{align*}
&v_{1\varepsilon}(t)\\
&=\varepsilon^{-\frac{1}{\alpha}}\int_0^tP_{\varepsilon^{-\frac{1}{\alpha}}\tau}\left(F(t-\tau)-F(t)\right)d\tau\\
&=\varepsilon^{-\frac{1}{\alpha}}\int_0^{r\varepsilon^{-\frac{1}{\alpha}}}P_{\varepsilon^{-\frac{1}{\alpha}}\tau}\left(F(t-\tau)-F(t)\right)d\tau+\varepsilon^{-\frac{1}{\alpha}}\int_{r\varepsilon^{-\frac{1}{\alpha}}}^tP_{\varepsilon^{-\frac{1}{\alpha}}\tau}\left(F(t-\tau)-F(t)\right)d\tau\\
&=\int_0^rP_\tau\left(F\left(t-\varepsilon^{\frac{1}{\alpha}}\right)-F(t)\right)d\tau+\varepsilon^{-\frac{1}{\alpha}}\int_{r\varepsilon^{-\frac{1}{\alpha}}}^tP_{\varepsilon^{-\frac{1}{\alpha}}\tau}\left(F(t-\tau)-F(t)\right)d\tau\\
&=:v_{1\varepsilon}^{(1)}(t)+v_{1\varepsilon}^{(2)}(t)
\end{align*}
By the continuity of $F(t)$, for any given $\rho>0$, we can choose $\varepsilon$ small enough such that
\[ \left\lVert F\left(t-\varepsilon^{\frac{1}{\alpha}}\right)-F(t)\right\rVert\lesssim\frac{\rho}{r^\alpha}, \]
then from $(\ref{lemma linear new3 proof equation3})$ it follows that
\begin{equation}
\left\lVert v_{1\varepsilon}^{(1)}(t)\right\rVert_H\lesssim\int_0^r\tau^{\alpha-1}\left\lVert F\left(t-\varepsilon^{\frac{1}{\alpha}}\right)-F(t)\right\rVert d\tau\lesssim\rho.
\label{8.30 Asymptotic behaviours theorem3 proof equation1}
\end{equation}
For $v_{1\varepsilon}^{(2)}(t)$, we have, by a slightly careful calculation, that
\begin{align*}
&\left\lVert v_{1\varepsilon}^{(2)}(t)\right\rVert_H\\
&\leq\left\lVert\varepsilon^{-\frac{1}{\alpha}}\int_{r\varepsilon^{-\frac{1}{\alpha}}}^tP_{\varepsilon^{-\frac{1}{\alpha}}\tau}^l\left(F(t-\tau)-F(t)\right)d\tau\right\rVert_H+\left\lVert\varepsilon^{-\frac{1}{\alpha}}\int_{r\varepsilon^{-\frac{1}{\alpha}}}^tP_{\varepsilon^{-\frac{1}{\alpha}}\tau}^h\left(F(t-\tau)-F(t)\right)d\tau\right\rVert_H\\
&\lesssim\left\lVert\varepsilon^{-\frac{1}{\alpha}}\int_{r\varepsilon^{-\frac{1}{\alpha}}}^tP_{\varepsilon^{-\frac{1}{\alpha}}\tau}^l\left(F(t-\tau)-F(t)\right)d\tau\right\rVert_H+\left\lVert\varepsilon\int_{r\varepsilon^{-\frac{1}{\alpha}}}^t\tau^{-\alpha-1}\mathbf{A}_{\varepsilon^{-\frac{1}{\alpha}}\tau}^{-2}\left(F(t-\tau)-F(t)\right)d\tau\right\rVert_H\\
&+\left\lVert\varepsilon^{-\frac{1}{\alpha}}\int_{r\varepsilon^{-\frac{1}{\alpha}}}^tR_{\varepsilon^{-\frac{1}{\alpha}}\tau}^P\left(F(t-\tau)-F(t)\right)d\tau\right\rVert_H\\
&\lesssim\left\lVert\varepsilon^{-\frac{1}{\alpha}}\int_{r\varepsilon^{-\frac{1}{\alpha}}}^t\chi_{\varepsilon^{-\frac{1}{\alpha}}\tau}b\left(\varepsilon^{-\frac{1}{\alpha}}\tau,\xi\right)U^{-1}\left(F(t-\tau)-F(t)\right)d\tau\right\rVert_{L^2(\Omega)}\\
&+\left\lVert\varepsilon\int_{r\varepsilon^{-\frac{1}{\alpha}}}^t\tau^{-\alpha-1}a(\xi)^{-2}\chi_{\varepsilon^{-\frac{1}{\alpha}}\tau}^cU^{-1}\left(F(t-\tau)-F(t)\right)d\tau\right\rVert_{L^2(\Omega)}\\
&+\left\lVert\varepsilon^2\int_{r\varepsilon^{-\frac{1}{\alpha}}}^t\tau^{-2\alpha-1}\left\lvert a(\xi)\right\rvert^{-3}\chi_{\varepsilon^{-\frac{1}{\alpha}}\tau}^cU^{-1}\left(F(t-\tau)-F(t)\right)d\tau\right\rVert_{L^2(\Omega)}\\
&\lesssim\int_{r\varepsilon^{-\frac{1}{\alpha}}}^t\tau^{-1}\left\lVert\left\lvert a(\xi)\right\rvert^{-1}U^{-1}\left(F(t-\tau)-F(t)\right)\right\rVert_{L^2(\Omega)}d\tau\lesssim r^{-1}\varepsilon^{\frac{1}{\alpha}}\left\lVert F\right\rVert_{L_t^\infty D(A^{-1})}.
\end{align*}
We can choose $r$ large enough and $\varepsilon$ small enough such that
\begin{equation}
\left\lVert v_{1\varepsilon}^{(2)}(t)\right\rVert_H\lesssim r^{-1}\varepsilon^{\frac{1}{\alpha}}\left\lVert F\right\rVert_{L_t^\infty D(A^{-1})}\lesssim\rho.
\label{8.30 Asymptotic behaviours theorem3 proof equation2}
\end{equation}
Combining $(\ref{8.30 Asymptotic behaviours theorem3 proof equation1})$ and $(\ref{8.30 Asymptotic behaviours theorem3 proof equation2})$ we obtain that for any given $\rho$ we can choose $r$ large enough and $\varepsilon$ small enough such that $\left\lVert v_{1\varepsilon}(t)\right\rVert_H\lesssim\rho$. Thus $v_{1\varepsilon}(t)\to0$ as $\varepsilon\to0$ uniformly on $[\delta,T]$ and then $(\ref{8.30 Asymptotic behaviours theorem3 equation2})$ holds.
\end{proof}
\end{appendices}
\section{Declarations}
\noindent\textbf{Conflict of interest}\quad The authors declare that they have no conflict of interest.
\bibliography{Reference}
\end{document}